\documentclass[production,11pt]{style}\usepackage{latexsym}\usepackage{amsmath}\usepackage{amsthm}\usepackage{xspace}\usepackage{enumerate}\usepackage{amsfonts}\usepackage{amscd}\usepackage{syntonly}\usepackage{amssymb}\usepackage{fancyhdr}

\swapnumbers

\newtheorem{thm}{Theorem}
\newtheorem{prop}[thm]{Proposition}
\newtheorem{lemma}[thm]{Lemma}
\newtheorem{claim}{Claim}
\newtheorem{defi}[thm]{Definition}
\newtheorem{rmk}[thm]{Remark}
\newtheorem{cor}[thm]{Corollary}

\newtheorem{conj}[thm]{Conjecture}

\newenvironment{pf}[1][Proof.]{\noindent \emph{#1.}}{}
\newenvironment{enui}{\begin{enumerate}[(i)]}{\end{enumerate}}

\def\inj{\hookrightarrow}
\def\coker{\operatorname{coker}}
\def\M{\operatorname{\mathcal{M}}}
\def\ev{\operatorname{ev}}
\def\barev{\overline{\operatorname{ev}}}
\def\torsion{\operatorname{torsion}}
\def\QH{\operatorname{QH}}
\def\GW{\operatorname{GW}}
\def\Qk{\operatorname{Q}\!\kappa}
\def\End{\operatorname{End}}

\def\ind{{\operatorname{ind}}}
\def\D{{\mathcal{D}}}
\def\DD{{\mathcal{\widetilde{D}}}}

\def\Ad{{\operatorname{Ad}}}
\def\cont{\supseteq}
\newcommand{\hhat}[1]{\widehat{#1}}
\def\Hat{{\widehat{\phantom{.}}}}
\newcommand\unhat[1]{{\overset{\vee}{#1}}}
\def\Unhat{\unhat{\phantom{.}}}

\def\B{\mathcal{B}}
\def\BB{\widetilde{\mathcal{B}}}
\def\E{\mathcal{E}}
\def\G{\mathcal{G}}
\def\Int{{\operatorname{int}}}

\def\const{\equiv}

\def\disj{\coprod}

\def\then{\Longrightarrow}

\def\im{{\operatorname{im}}}

\def\Aut{{\operatorname{Aut}}}

\def\id{{\operatorname{id}}}

\def\nn{{\nonumber}}
\newcommand\wt[1]{{\widetilde{#1}}}
\newcommand{\BAR}[1]{{\overline{#1}}}
\def\al{{\alpha}}
\def\be{\beta}
\def\Ga{\Gamma}

\def\de{\delta}
\def\eps{\varepsilon}
\def\Om{\Omega}
\def\om{\omega}
\def\ka{\kappa}
\def\Lam{{\Lambda}}
\def\lam{{\lambda}}
\def\La{\Delta}
\def\Si{\Sigma} 
\def\si{\sigma}

\def\ze{\zeta}
\renewcommand\phi{\varphi}
\def\na{\nabla}

\def\U{\operatorname{U}}
\def\X{\mathcal{X}}
\def\XX{\widetilde{\mathcal{X}}}
\def\Y{\mathcal{Y}}
\def\YY{\widetilde{\mathcal{Y}}}
\newcommand{\N}{\mathbb{N}}
\newcommand{\Z}{\mathbb{Z}}

\newcommand{\R}{\mathbb{R}}
\renewcommand\S{\mathcal{S}}
\def\SS{\widetilde{\mathcal{S}}}
\def\C{\mathbb C}

\def\g{\mathfrak g}
\def\A{\mathcal A}
\def\Lie{\operatorname{Lie}}
\def\EG{{\operatorname{EG}}}
\def\pr{{\operatorname{pr}}}

\def\sub{\subseteq}

\def\x{\times}
\def\wo{\setminus}
\def\one{\mathbf{1}}
\def\iso{\cong}
\def\d{\partial}

\def\lan{\langle}
\def\ran{\rangle}

\def\ie{i.e.\xspace}
\def\eg{e.g.\xspace}

\def\loc{{\operatorname{loc}}}
\def\Co{\mathcal{C}}
\def\wone{{\bigwedge}^{\!1}}
\def\wtwo{{\bigwedge}^{\!2}}
\def\wk{{\bigwedge}^{\!k}}
\def\wzeroone{{\bigwedge}^{\!0,1}}

\def\pt{\operatorname{pt}}
\def\J{\mathcal{J}}

\author{Fabian Ziltener}
\title{A Quantum Kirwan Map, I: Fredholm Theory}

\begin{document}

{\bf This article has been merged with arXiv:1106.1729. The new article is:\\

A Quantum Kirwan Map: Bubbling and Fredholm Theory for Symplectic Vortices over the Plane, arXiv:1209.5866}\\

\begin{abstract}Consider a Hamiltonian action of a compact connected Lie group $G$ on an aspherical symplectic manifold $(M,\om)$. Under some assumptions on $(M,\om)$ and the action, D.~A.~Salamon conjectured that counting gauge equivalence classes of symplectic vortices on the plane $\R^2$ gives rise to a quantum deformation $\Qk_G$ of the Kirwan map. This article is the first of three, whose goal is to define $\Qk_G$ rigorously. Its main result is that the vertical differential of the vortex equations over $\R^2$ (at the level of gauge equivalence) is a Fredholm operator of a specified index. Potentially, the map $\Qk_G$ can be used to compute the quantum cohomology of many symplectic quotients. Conjecturally it also gives rise to quantum generalizations of non-abelian localization and abelianization (see \cite{WZ}). 
\end{abstract}

\maketitle

\tableofcontents

\section{Main result and motivation}\label{sec:main}
Let $(M,\om)$ be a symplectic manifold without boundary, and $G$ a compact connected Lie group with Lie algebra $\g$. We fix a Hamiltonian action of $G$ on $M$ and an (equivariant) moment map $\mu:M\to\g^*$. Assume that the following hypothesis is satisfied:\\

\noindent {(\bf H)} \textit{$G$ acts freely on $\mu^{-1}(0)$ and the moment map $\mu$ is proper.}\\

Then the symplectic quotient $\big(\bar M:=\mu^{-1}(0)/G,\bar\om\big)$ is well-defined, smooth and closed. We denote by $H^G_*(M,\Z)$ ($H_G^*(M,\Z)$) the equivariant (co-)homology of $M$ with integer coefficients (equipped with the cup product $\smile$), $H^G_*(M):=H^G_*(M,\Z)/\torsion$ etc., by $\ka_G:H_G^*(M)\to H^*(\bar M)$ the Kirwan map, and by $[\om-\mu]^G\in H_G^2(M)$ the class of $\om-\mu$ in the Cartan model. We define $\Lam_\om^\mu$ to be the set of maps $\lam:H_2^G(M)\to \Z$ such that 
\[\big|\big\{B\in H^G_2(M)\,\big|\,\lam(B)\neq0,\lan[\om-\mu]^G,B\ran\leq C\big\}\big|<\infty,\quad\forall C\in\R,\]
we equip $\Lam_\om^\mu$ with the convolution product $\cdot$, and call the triple $(\Lam_\om^\mu,+,\cdot)$ the \emph{equivariant Novikov ring}. We denote by $(\QH(\bar M,\bar \om),\bar *)$ the (small) quantum cohomology of $(\bar M,\bar\om)$ with coefficients in $\Lam_\om^\mu$, and by $\J^G(M,\om)$ the space of $G$-invariant and $\om$-compatible almost complex structures on $M$. For $x\in M$ we denote by $L_x:\g\to T_xM$ the infinitesimal action. We call $(M,\om)$ \emph{equivariantly convex at $\infty$} iff there exists a proper $G$-invariant $f\in C^\infty(M,[0,\infty))$, $J\in\J^G(M,\om)$ and $C\in[0,\infty)$ such that 
\[\om(\na_v\na f(x),Jv)-\om(\na_{Jv}\na f(x),v)\geq0,\quad df(x)JL_x\mu(x)\geq0,\]
for every $x\in f^{-1}([C,\infty))$ and $0\neq v\in T_xM$. Here $\na$ denotes the Levi-Civita connection of the metric $\om(\cdot,J\cdot)$. A symplectic manifold $(X,\si)$ is called semi-positive iff for every $B\in \pi_2(X)$ the conditions $\lan[\om],B\ran>0$ and $c_1(X,\si)\geq3-\dim X/2$ imply that $c_1(X,\si)\geq0$ (see \cite{MSJ}). Note that this holds for example if $(X,\si)$ is weakly monotone or $\dim X\leq6$. We call $(X,\si)$ aspherical iff $\int u^*\si=0$ for every $u\in C^\infty(S^2,X)$.
\begin{conj}\label{conj:QK G} Assume that (H) holds, $(M,\om)$ is equivariantly convex at $\infty$ and aspherical, and $(\bar M,\bar\om)$ is semi-positive (see \cite{MSJ}). Then there exists a $\Lam_\om^\mu$-algebra homomorphism
\begin{equation}\label{eq:phi H G M}\phi:H^*_G(M)\otimes\Lam_\om^\mu\to \QH^*(\bar M,\bar\om)\end{equation}
of the form $\phi=\ka_G\otimes \id+\sum_{0\neq B}\phi^B\otimes e^B\cdot$. 
\end{conj}
This conjecture (without specification of the quantum coefficient ring involved) was formulated by D.~A.~Salamon. The idea of proof outlined below is also due to him. The present article is part of a project whose goal is to make this idea rigorous and hence prove the conjecture. As an example, consider $M:=\R^{2n}$ with the standard structure $\om:=\om_0$, and a linear action of $G$. Then $(M,\om)$ is equivariantly convex at $\infty$ (see Example 2.8 in \cite{CGMS}) and aspherical. Assume that $G$ is the torus $\R^k/\Z^k$, and let $w^1,\ldots,w^n\in\g^*\iso(\R^k)^*$ be the weights of the action and $\mu:\R^{2n}\iso\C^n\to\R^k$ be given by $\mu(z):=\tau-\pi\sum_i|z^i|^2w^i$, for some $\tau\in\g^*$. Then $\mu$ is proper if and only if there exists $\xi\in\g$ such that $w^i(\xi)>0$, for $i=1,\ldots,n:=\dim M/2$ (see Proposition 4.14 in \cite{GGK}). The action of $G$ on $\mu^{-1}(0)$ is free if and only if for every subset $I\sub\{1,\ldots,n\}$ the following holds. If there exist $a_i>0$, for $i\in I$, such that $\sum_{i\in I}a_iw^i=0$ then $w^i$, for $i\in I$, generate the integral lattice in $\g^*$ (see Lemma 5.20 in \cite{GGK}). The quotient $(\bar M,\bar\om)$ is semi-positive if for every $\xi\in\g$ the conditions $\lan\sum_{i=1}^nw^i,\xi\ran\geq3-n+k$, $\lan\tau,\xi\ran>0$ imply $\lan\sum_{i=1}^nw^i,\xi\ran\geq0$. Assume that the hypotheses of Conjecture \ref{conj:QK G} are satisfied, the action of $G$ on $M$ is monotone (and hence $(\bar M,\bar\om)$ is monotone), and $H^*_G(M)$ is generated by classes of degree less than twice the minimal Maslov number of this action. With these hypotheses R.~Gaio and D.~A.~Salamon \cite{GS} proved a version of the conjecture involving the Novikov ring of $(\bar M,\bar\om)$. This was used by K.~Cieliebak and D.~A.~Salamon \cite{CS} to compute $\QH^*(\bar M,\bar\om)$ for monotone torus actions on $\R^{2n}$ with Maslov number at least 2. A potential application of Conjecture \ref{conj:QK G} is to extend these computations to the more general setting of this conjecture. 

The idea of proof of Conjecture \ref{conj:QK G} is to define $\phi$ by counting symplectic vortices on the plane. To explain this, we fix a $G$-invariant inner product $\lan\cdot,\cdot\ran_\g$ on $\g$ and identify $\g^*$ with $\g$ via $\lan\cdot,\cdot\ran_\g$. Let $(\Si,\om_\Si,j)$ be a real surface equipped with an area form and a compatible complex structure, and $\pi:P\to\Si$ a principal $G$-bundle. We denote by $C^\infty_G(P,M)$ the space of smooth equivariant maps from $P$ to $M$, and by $\A(P)$ the space of smooth connections on $P$. A (symplectic) vortex is a solution $(u,A)\in C^\infty_G(P,M)\x\A(P)$ of 
\begin{equation}\label{eq:vort P}\bar \d_{J,A}(u)=0,\quad F_A+(\mu\circ u)\om_\Si=0.\end{equation}
Here $\bar\d_{J,A}(u)$ denotes the complex anti-linear part of $d_Au:=du+L_uA$, which we think of as a one-form on $\Si$ with values in the complex vector bundle $TM^u:=(u^*TM)/G\to\Si$. Similarly, we view the curvature $F_A$ of $A$ as a two-form on $\Si$ with values in the adjoint bundle $\g_P:=(P\x\g)/G\to\Si$. Finally, we view $\mu\circ u$ as a section of $\g_P$. The vortex equations (\ref{eq:vort P}) were discovered by K.~Cieliebak, A.~R.~Gaio and D.~A.~Salamon \cite{CGS}, and independently by I.~Mundet i Riera \cite{Mu1,Mu2}. 

The energy density and energy of a $w:=(u,A)\in C^\infty_G(P,M)\x\A(P)$ are given by $e_w:=\frac12\big(|d_Au|^2+|F_A|^2+|\mu\circ u|^2\big)$ and $E(w):=\int_\Si e_w\om_\Si$. Consider $\Si:=\R^2$, equipped with the standard area form $\om_{\R^2}:=\om_0$ and complex structure $j:=i$, and $P:=\R^2\x G$. Let $B\in H^G_2(M)$. The group $\G:=\G(P)$ of smooth gauge transformations on $P$ acts on the set of solutions of (\ref{eq:vort P}). We denote by $\M_B$ the set of gauge (equivalence) classes of vortices $w:=(u,A)$ on $P$ for which $E(w)<\infty$, $\BAR{u(P)}\sub M$ is compact and $w$ represent the equivariant homology class $B$ (see \cite{ZiPhD}). We denote by $\EG$ a contractible topological space on which $G$ acts continuously and freely. There are natural evaluation maps $\ev_z:\M_B\to (M\x\EG)/G$ (at $z\in \R^2$) and $\barev_\infty:\M_B\to \bar M$ (at $\infty\in \R^2\cup\{\infty\}$) (see again \cite{ZiPhD}). Heuristically, for $\al\in H_G^*(M)$ and $\bar\be\in H^*(\bar M)$, we define 
\begin{equation}\label{eq:QK G B}\Qk_G^B(\al,\bar\be):=\int_{\M_B}\ev_0^*\al\smile\barev_\infty^*\bar\be.\end{equation}
We fix dual bases $(\bar e_i)_{i=1,\ldots,N}$ and $(\bar e^*_i)_{i=1,\ldots,N}$ of $H^*(\bar M)$, in the sense that $\int_{\bar M}\bar e_i\smile \bar e^*_j=\de_{ij}$. The idea is now to define a map $\phi:=\Qk_G$ as in (\ref{eq:phi H G M}) by
\[\Qk_G(\al):=\sum_{i=1,\ldots,N,\,B\in H^G_2(M)}\Qk_G^B(\al,\bar e_i)\bar e^*_i\otimes e^B.\]

The goal of the present and two subsequent papers \cite{ZiII,ZiIII} is to make this definition rigorous: In the present article I introduce weighted Sobolev spaces $\X_w^{p,\lam}$ and $\Y_w^{p,\lam}$ that will serve as model spaces for a Banach manifold $\B^p_\lam$ of gauge classes of pairs $(u,A)$ and the fiber of a Banach bundle $\E^p_\lam\to\B^p_\lam$. Furthermore, I show that the vertical differential of the equations (\ref{eq:vort P}) (viewed as a section $\S$ of $\E^p_\lam$) is a Fredholm map between $\X_w^{p,\lam}$ and $\Y_w^{p,\lam}$. In \cite{ZiII} the notion of a stable map of vortices on $\R^2$ and (pseudo-)holomorphic spheres in $\bar M$ is defined, and a bubbling result for vortices over $\R^2$ is proven. The hypotheses of equivariant convexity and asphericity are needed here. The former forces vortices on $\R^2$ to stay in a fixed compact subset of $M$, and the latter rules out bubbling of holomorphic spheres in $M$. The present paper and \cite{ZiII} are based on my Ph.D.-thesis \cite{ZiPhD}. In \cite{ZiIII}, I define an atlas for $\B^p_\lam$ consisting of charts with targets open subsets of the spaces $\X_w^{p,\lam}$ (and a similar atlas for $\E^p_\lam$). The charts are based on the choice of a $G$-invariant Riemannian metric on $M$ whose exponential map along $\mu^{-1}(0)$ is compatible with the moment map $\mu$. (This is needed to make the transition maps well-defined.) Furthermore, I show how to perturb $\S$ in order to make it transverse to the zero section. Combining all these results, it follows that $\S^{-1}(0)\sub\B^p_\lam$ is a smooth finite dimensional submanifold, and the evaluation map $\ev_0\x\barev_\infty$ is a pseudo-cycles. Here the semi-positivity of $(\bar M,\bar\om)$ is needed. The formula (\ref{eq:QK G B}) is then made rigorous as an intersection number of pseudo-cycles. 

In order to show that $\Qk_G$ intertwines $\smile$ with $\bar *$, it suffices to prove that
\begin{equation}\label{eq:ring}\big\lan\Qk_G^B(\al_1\smile\al_2),\bar a\big\ran=\big\lan\big(\Qk_G(\al_1)\bar *\Qk_G(\al_2)\big)_B,\bar a\big\ran,
\end{equation}
for every $B\in H_2^G(M)$, $\al_1,\al_2\in H_G^*(M)$ and $\bar a\in H_*(\bar M)$. The idea for proving this is the following. We choose ``oriented submanifolds'' $X_1,X_2\sub (M\x\EG)/G$ that are ``Poincar\'e dual'' to $\al_1$ and $\al_2$, and an oriented submanifold $\bar X\sub\bar M$ representing $\bar a$. (To make this rigorous one has to pass to some finite dimensional approximation of $\EG$, a compact submanifold of $M$ with boundary and rational multiples of $\al_1,\al_2$ and $\bar a$.) Consider the marked points $z_\nu^\pm:=\pm\nu$, $z_\nu^\infty:=\infty$, and a sequence of gauge classes of vortices $W_\nu\in\M_B$, such that $\ev_{z_\nu^+}(W_\nu)\in X_1$, $\ev_{z_\nu^-}(W_\nu)\in X_2$, and $\barev_\infty(W_\nu)\in\bar X$. By the main result of \cite{ZiII} a subsequence of $W_\nu$ converges to a stable map of (gauge classes of) vortices on $\R^2$ and holomorphic spheres in $\bar M$. In the transverse case, this map consists of two classes $W^1$ and $W^2$ of vortices on $\R^2$, each equipped with a marked point $z_i\in\C$, and a holomorphic map $\bar u:S^2\to\bar M$, equipped with a marked point $z_\infty\in S^2$. $W^i$ is attached to $\bar u$ at the nodal point $\infty\in\R^2\cup\{\infty\}$, for $i=1,2$. The total homology class of the stable map equals $B$, $\ev_{z_i}(W^i)\in X_i$, for $i=1,2$, and $\bar u(z_\infty)\in\bar X$. The number of such stable maps equals the right hand side of (\ref{eq:ring}). 

Consider now the marked points $z_\nu^\pm:=\pm1/\nu$, $z_\nu^\infty:=\infty$, and a sequence $W_\nu\in\M_B$, such that $\ev_{z_\nu^+}(W_\nu)\in X_1$, $\ev_{z_\nu^-}(W_\nu)\in X_2$ and $\barev_\infty(W_\nu)\in\bar X$. Using \cite{ZiII} again, a subsequence of $W_\nu$ converges to a stable map of vortices on $\R^2$ and spheres in $\bar M$. In the transverse case, this stable map consists of a single class $W\in\M_B$, satisfying $\barev_\infty(W)\in\bar X$, and a ghost bubble $u\const\pt\in X_1\cap X_2$ that is attached to $W$ at some point in $\R^2$ and contains two marked points. The number of such stable maps equals the left hand side of (\ref{eq:ring}). Equality (\ref{eq:ring}) follows by combining this with the argument of the last paragraph and a gluing argument for vortices on $\R^2$ and spheres in $\bar M$. The gluing result is part of my future research. 

A similar argument involving an adiabatic limit in the vortex equations over $S^2$ will show that $\Qk_G$ intertwines the genus 0 symplectic vortex invariants with the Gromov-Witten invariants of $(\bar M,\bar\om)$. The idea of proof of Conjecture \ref{conj:QK G} presented here is different from the construction used by Gaio and Salamon in \cite{GS}. They use an adiabatic limit argument, which fails in the more general situation considered here, because of bubbling off of vortices on $\R^2$. 

Assume now just that (H) holds and $(M,\om)$ is equivariantly convex at $\infty$. We denote by $\GW_G(M,\om)$ the $G$-equivariant Gromov-Witten theory of $(M,\om)$. In joint work with Christopher Woodward \cite{WZ} we interpret $\Qk_G$ as a morphism of cohomological field theories between $\GW_G(M,\om)$ and $\GW(\bar M,\bar\om)$. We formulate ``functoriality'' for $\GW_\bullet$  under reduction in stages. We also conjecture quantum generalizations involving $\Qk_G$ of non-abelian localization and abelianization. 

For a vector bundle $E\to X$ and $k\in\N\cup\{0\}$ we denote by $\Ga(E)$ the space of its smooth sections and by $\wk(E)$ the bundle of $k$-forms on $X$ with values in $E$. For an almost complex manifold $X$ and a complex vector bundle $E\to X$ is we denote by $\wzeroone(E)\to X$ the bundle of anti-linear one-forms on $X$ with values in $E$. We equip the bundle $TM^u=(u^*TM)/G\to\R^2$ with the complex structure induced by $J$. We define 
\begin{eqnarray}\nn&\BB:=C_G^\infty(P,M)\x\A(P),&\\
\nn &\wt\E:=\big\{(w;\ze')\,\big|\,w\in\BB,\,\ze'\in \Ga\big(\wzeroone(TM^u)\x\wtwo(TM^u)\big)\big\},&\\ 
\nn&\SS:\BB\to\wt\E,\quad \SS(u,A):=\big(\bar \d_{J,A}(u),F_A+(\mu\circ u)\om_\Si\big).\end{eqnarray}
The group $\G$ acts naturally on $\BB$ and $\wt\E$. We denote by $\B:=\BB/\G$ and $\E:=\wt\E/\G$ the quotients. The map $\SS$ is $\G$-equivariant, and hence induces a map $\S:\B\to\E$. Note that $\S^{-1}(0)\sub \B$ is the set of gauge classes of vortices. Assume that the action of $\G$ on $\BB$ is free. (This is satisfied for $\Si:=\R^2$, $\om_\Si:=\om_{\R^2}$ and $j:=i$ under hypothesis (H), see Lemma \ref{le:free} below.) Then heuristically, $\B$ is an infinite dimensional manifold, $\E$ is an infinite dimensional vector bundle over $\B$, and $\S$ is a smooth section of $\E$. 

Assume that $W\in\S^{-1}(0)$. Then formally, there is a canonical map $T:T_{(W,0)}\E\to\E_W$, where $\E_W\sub\E$ denotes the fiber over $W$. The vertical differential of $\S$ at $W$ is given by $d^V\S(W)=T\,d\S(W):T_W\B\to\E_W$. Let $w:=(u,A)\in W$ be a representative. We denote by $L_w:\Lie(\G)\to T_w\BB$ the infinitesimal action of $\G$ on $\BB$. Formally, $T_w\BB=\Ga\big(TM^u\oplus\wone(\g_P)\big)$, $\Lie(\G)=\Ga(\g_P)$, and $L_w\xi=(L_u\xi,-d_A\xi)$. 

Assume that $\BB$ and $\Lie(\G)$ are equipped with a $\G$-invariant Riemannian metric and $\G$-invariant inner product respectively. For $w\in\BB$ we denote by $L_w^*:T_w\BB\to\Lie(\G)$ the adjoint of $L_w$. Then formally, the tangent space $T_W\B$ is the quotient of the disjoint union of $\ker L_w^*$, with $w$ ranging over all representatives of $W$, by the linearized action of $\G$. We now equip $M$ with the Riemannian metric $\om(\cdot,J\cdot)$, and $TM^u$ with the induced bundle metric. The bundle $\g_P$ inherits a bundle metric from $\lan\cdot,\cdot\ran_\g$. We equip $\Lie(\G)$ and $T_w\BB$ with the corresponding $L^2$ inner products. Then $L_w^*$ is given by 
\begin{equation}\label{eq:L w *} L_w^*(v,\al)=L_u^*v-d_A^*\al.
\end{equation}
Here we view $L_u$ as a map from $\g_P$ to $TM^u$, and we denote by $L_u^*:TM^u\to\g_P$ the adjoint of $L_u$, and $d_A^*=-*d_A*$, where $*$ denotes the Hodge-star operator on sections of $\g_P$ with respect to the metric $\om_\Si(\cdot,j\cdot)$. 

The Levi-Civita connection $\na$ of $\om(\cdot,J\cdot)$ and $A$ induce a connection $\na^A$ on $TM^u$ (see Section \ref{sec:back}). For $w\in \BB$ consider the operator $\D_w:\ker L_w^*\to \wt\E_w$, 
\begin{equation}\label{eq:D w} \D_w(v,\al):=\left(\begin{array}{c}\big(\na^Av+L_u\al\big)^{0,1}-\frac12J(\na_vJ)(d_Au)^{1,0} \\
d_A\al+\om_\Si\, d\mu(u)v
\end{array}\right),\end{equation}
For $W\in\S^{-1}(0)$ the map $d^V\S(W):T_W\B\to \E_W$ is given by $d^V\S(W)\G^*(v,\al)=\G^*\D_w(v,\al)$. (This follows for example from \cite{CGMS}, formula (23), p.~27.) 

Consider now the case $\Si:=\R^2,\om_{\R^2}:=\om_0$ and $j:=i$. The purpose of this article is to find a Banach space setup in which the vertical differential $d^V\S(W)$ is Fredholm. To this end, let $n\in\N$, $p\in[1,\infty]$, $\lam\in \R$, and  $f:\R^n\to\R$ be a measurable function. We denote $\Vert f\Vert_p:=\Vert f\Vert_{L^p(\R^n)}\in[0,\infty]$, and define $\Vert f\Vert_{p,\lam}:=\big\Vert f(1+|\cdot|^2)^\frac\lam2\big\Vert_p$. Assume now that $p>2$. We denote by $W^{1,p}_{\loc, G}(P,M)$ and $\A^{1,p}_\loc(P)$  the spaces of $G$-equivariant maps from $P\to M$ and connections on $P$, of class locally $W^{1,p}$. We abbreviate $\wt\B^p_\loc:=W^{1,p}_{\loc, G}(P,M)\x\A^{1,p}_\loc(P)$, and we define
\begin{equation}\label{eq:BB p lam} \BB^p_\lam:=\big\{(u,A)\in \wt\B^p_\loc\,\big|\,\BAR{u(P)}\textrm{ compact, }\Vert\sqrt{e_{(u,A)}}\Vert_{p,\lam}<\infty\big\}.
\end{equation}
Furthermore, we denote by $\G^{2,p}_\loc(P)$ the group of locally $W^{2,p}$ gauge transformations on $P$, and we define $\B^p_\lam:=\BB^p_\lam/\G^{2,p}_\loc(P)$.  

Let $w:=(u,A)\in\BB^p_\lam$ be a smooth pair. $\na,A$ and the Levi-Civita connection of the standard metric $g_{\R^2}$ on $\R^2$ induce a linear connection $\na^A$ on $TM^u\oplus\wone(\g_P)$. For $\ze:=(v,\al)\in W^{1,p}_\loc\big(TM^u\oplus\wone(\g_P)\big)$ we define
\begin{equation}\label{eq:ze w p lam}\Vert\ze\Vert_{w,p,\lam}:=\Vert\ze\Vert_\infty+\big\Vert|\na^A\ze|+|d\mu(u)v|+|\al|\big\Vert_{p,\lam}\in[0,\infty].
\end{equation}
Here the pointwise norms are taken with respect to the metrics $\om(\cdot,J\cdot)$ and $g_{\R^2}$, and $\lan\cdot,\cdot\ran_\g$. We define 
\begin{align}
  \label{eq:XX w}\X_w^{p,\lam}&:=\big\{\ze\in W^{1,p}_\loc\big(TM^u\oplus\wone(\g_P)\big)\,\big|\,L_w^*\ze=0,\,\Vert\ze\Vert_{w,p,\lam}<\infty\big\},\\ 
\label{eq:YY w}\Y_w^{p,\lam}&:=\big\{\ze'\in L^p_\loc\big(\wzeroone(T\C,TM^u)\oplus\wtwo(\g_P)\big)\,\big|\,\Vert\ze'\Vert_{p,\lam}<\infty\big\}. 
\end{align}
Here $L_w^*$ is as in (\ref{eq:L w *}). In \cite{ZiIII}, the set $\B^p_\lam$ will be equipped with a Banach manifold structure, such that for every $W\in \B^p_\lam$ admitting a smooth representative $w\in\BB^p_\lam$, the tangent space $T_W\B^p_\lam$ can be identified with $\X_w^{p,\lam}$. Furthermore, the spaces $\Y_w^{p,\lam}$ will be identified with the fibers of a Banach bundle $\E^p_\lam\to\B^p_\lam$. \\

{\bf From now on throughout this article, we assume that hypothesis (H) is satisfied.}\\

We denote by $m(w)$ the Maslov index of $w$, see Section \ref{sec:back}. The main result of this article is the following.
\begin{thm}\label{thm:Fredholm} Assume that $\dim M>2\dim G$. Let $p>2$, $\lam>1-2/p$ and $w:=(u,A)\in\BB^p_\lam$ be a smooth pair. Then the following statements hold.
\begin{enui}\item \label{thm:Fredholm:X Y} The normed vector spaces $\X^{p,\lam}_w$ and $\Y^{p,\lam}_w$ are complete. 
\item\label{thm:Fredholm:DDD} If $1-2/p<\lam<2-2/p$ then the operator $\D^{p,\lam}_w:\X^{p,\lam}_w\to \Y^{p,\lam}_w$ given by the formula (\ref{eq:D w}) is well-defined and Fredholm of real index $\ind\D^{p,\lam}_w=\dim M-2\dim G+2m(w)$. 
\end{enui}
\end{thm}
The condition $1-2/p<\lam<2-2/p$ in this theorem captures the geometry of finite energy vortices. More precisely, let $w:=(u,A)\in\B^p_\loc$ be a finite energy vortex such that $\BAR{u(P)}\sub M$ is compact. Then for every $\eps>0$ there exists a constant $C$ such that $e_w(z)\leq C|z|^{-4+\eps}$, for every $z\in\R^2\wo B_1$ (see \cite{ZiA}, Corollary 4). It follows that $w\in\BB^p_\lam$ if $\lam<2-2/p$. This bound is sharp. To see this, let $\lam>2-2/p$, let $M:=S^2,\om:=\om_0,G:=\{e\}$ and $J:=i$, and consider the inclusion $u:\R^2\to S^2\iso\R^2\cup\{\infty\}$. 

On the other hand, every $w\in \BB^p_\lam$ satisfies $E(w)<\infty$ if $p>2$ and $\lam>1-2/p$. The latter condition is sharp. Namely, let $\lam< 1-2/p$, and consider $M:=S^2$ with the standard symplectic form $\om_0$, complex structure $J:=i$ and the action of the trivial group $G:=\{e\}$. We choose a number $2<a<3-2/p-\lam$ and a smooth map $u:\R^2\to S^2\iso\C\cup\{\infty\}$ such that $u(z)=|z|^a$, for $z\in\R^2\wo B_1$. Then $(u,0)\in \BB^p_\lam$ and $E(u,0)=\infty$. 

The condition $\lam<2-2/p$ is also needed for $\D^{p,\lam}_w$ to have the right Fredholm index. Namely, let $\lam>1-2/p$ be such that $\lam+2/p\not\in\Z$, and $w\in\BB^p_\lam$. Then the proof of Theorem \ref{thm:Fredholm} shows that $\D^{p,\lam}_w$ is Fredholm with index equal to $(2-k)(\dim M-2\dim G)+2m(w)$, where $k$ is the largest integer less than $\lam+2/p$. In particular, the index changes when $\lam$ passes the value $2-2/p$. 

The definition of the space $\X^{p,\lam}_w$ is natural, since it parallels the definition of $\BB^p_\lam$. Namely, by linearizing with respect to $u$ and $A$ the terms $d_Au,F_A$ and $\mu\circ u$ occurring in the energy density $e_w$, we obtain the terms $\na^A\ze,d\mu(u)v$ and $L_u\al$. These expressions occur in $\Vert\ze\Vert_{w,p,\lam}$, except for the factor $L_u$ in $L_u\al$. (It follows from (H) and Lemma \ref{le:si} (appendix) that this factor is irrelevant.) The expression $\Vert\ze\Vert_\infty$ is needed in order to make $\Vert\cdot\Vert_{w,p,\lam}$ non-degenerate. 

{\bf Remark.} \textit{Naively, one could define the domain of $\D_w$ to be the kernel of $L_w^*$ defined on the space of usual $W^{1,p}$-sections, and its target to consist of $L^p$-sections. However, then in general $\D_w$ would not have closed image, and hence is not Fredholm. Note also that the 0-th order terms $\al\mapsto(L_u\al)^{0,1}$ and $v\mapsto\om_0\,d\mu(u)v$ in (\ref{eq:D w}) are not compact (neither with $\X_w^{p,\lam}$ and $\Y_w^{p,\lam}$ defined as in (\ref{eq:XX w},\ref{eq:YY w}) nor the naive choices). The reason is that the Kondrachov compactness theorem fails on $\R^2$. Observe also that because of these terms, $\D_w$ is not well-defined if we choose spaces that look like the standard (weighted) Sobolev spaces in ``logarithmic'' coordinates $\tau+i\phi$ (with $e^{\tau+i\phi}=z\in\C\wo\{0\}$).}

The proof of Theorem \ref{thm:Fredholm} is based on a Fredholm result for the augmented vertical differential (Theorem \ref{thm:Fredholm aug}) and surjectivity of $L_w^*$ (Theorem \ref{thm:L w * R}). The proof of Theorem \ref{thm:Fredholm aug} has two main ingredients. The first one is a suitable complex trivialization of the bundle $TM^u\oplus\wone(\g_P)$. For $R$ large, $z\in\R^2\wo B_R$ and $p\in\pi^{-1}(z)\sub P$ this trivialization respects the splitting $T_{u(p)}M=(\im L^\C_{u(p)})^\perp\oplus\im L^\C_{u(p)}$, where $L_x^\C:\g\otimes\C\to T_xM$ denotes the complexified infinitesimal action, for $x\in M$. The second ingredient are two propositions stating that the standard Cauchy-Riemann operator $\d_{\bar z}$ and a related matrix differential operator are Fredholm maps between suitable weighted Sobolev spaces. These results are based on the analysis of weighted Sobolev spaces carried out by R.~B.~Lockhart and R.~C.~McOwen. (See \cite{Lockhart Hodge} and references therein.) Note that for a \emph{compact} Riemann surface $\Si$ without boundary, in \cite{CGMS} K.~Cieliebak et al.~ proved that the augmented vertical differential of the vortex equations is Fredholm.
\subsection*{Organization of the article} 
Section \ref{sec:back} contains some background about the connection $\na^A$ and the definition of the Maslov index $m(w)$. In Section \ref{subsec:reform} a Fredholm theorem for the augmented vertical differential (Theorem \ref{thm:Fredholm aug}), and an existence result for a right inverse for $L_w^*$ (Theorem \ref{thm:L w * R}) are stated. Furthermore, the main result is deduced from these results. Section \ref{subsec:proof:Fredholm aug} contains the core of the proof of Theorem \ref{thm:Fredholm aug}. Here the notion of a good complex trivialization is introduced and an existence result for such a trivialization is stated (Proposition \ref{prop:triv}). Furthermore, a result is formulated saying that every good trivialization transforms $\D_w$ into a compact perturbation of the direct sum of $\d_{\bar z}$ and a matrix operator (Proposition \ref{prop:X X w}). The results of Section \ref{subsec:proof:Fredholm aug} are proved in Section \ref{subsec:proofs}. In Section \ref{subsec:proof:thm:L w * R} Theorem \ref{thm:L w * R} is proved, using the existence of a right inverse for $d_A^*$ (Proposition \ref{prop:right}). Appendix \ref{sec:weighted} contains a Hardy-type inequality, which is used in the proof of Proposition \ref{prop:X X w}, some standard embedding and compactness results for weighted Sobolev spaces, and Fredholm results for the Cauchy-Riemann operator and a matrix valued operator on $\R^2$ (Propositions \ref{prop:d L L} and \ref{prop:d A B d} and Corollary \ref{cor:d L L}). In Appendix \ref{sec:proof:prop:right} Proposition \ref{prop:right} is proved. Appendix \ref{sec:aux} contains some other auxiliary results.
\subsection*{Acknowledgment} This article arose from my Ph.D.-thesis. I would like to thank my adviser, Dietmar A.~Salamon, for the inspiring topic. I highly profited from his mathematical insight. I am very much indebted to Chris Woodward for his interest in my work, for sharing his ideas with me, and for his continuous encouragement. It was he who coined the term ``quantum Kirwan map''. I would also like to thank Urs Frauenfelder and Kai Cieliebak for stimulating discussions. Finally, financial support of the dissertation from the Swiss National Science Foundation is gratefully acknowledged.

\section{Background and notation}\label{sec:back}

\subsection*{The connection $\na^A$}\label{subsec:na A}
Let $E\to M$ be a real (smooth) vector bundle. We denote by $\Co(E)$ the affine space of (smooth linear) connections on $E$. Let $\na^E\in\Co(E)$. Let $N$ be a smooth manifold, and $u:N\to M$ be a smooth map. We denote by $u^*E\to N$ the pullback bundle. The pullback connection $u^*\na^E\in\Co(u^*E)$ is uniquely determined by $(u^*\na^E)_vs\circ u=\na^E_{u_*v}s$, for every $v\in TN$ and every $s\in\Ga(E)$. Let $G$ be a Lie group, $\pi:P\to X$ a (right-)principal $G$-bundle, and $E\to P$ a $G$-equivariant vector bundle. Then the quotient $E/G$ has a natural structure of a vector bundle over $X$. Let now $E\to M$ be a $G$-equivariant vector bundle. We denote by $\Co^G(E)$ the space of $G$-invariant connections on $E$. We fix $A\in \A(P)$, $\na^E\in\Co^G(E)$, and $u\in C_G^\infty(P,M)$. We define $\wt\na^A\in\Co^G(u^*E)$ by $\wt\na^A_{\wt v}\wt s:=(u^*\na^E)_{\wt v-pA\wt v}\wt s$, for $\wt s\in\Ga(u^*E)$, $p\in P$, and $\wt v\in T_pP$. We denote $E^u:=(u^*E)/G\to X$. The connection $\wt\na^A$ is basic (\ie $G$-invariant and horizontal), hence there exists a unique $\na^A\in\Co(E^u)$ with the following property. Let $s\in\Ga(E^u)$ and $v\in TX$. We define $\na^A_vs:=G\cdot(p_0,\wt\na^A_{\wt v}\wt s)$, where $(p_0,\wt v)\in TP$ is such that $\pi_*\wt v=v$, and $\wt s\in\Ga(u^*E)$ is a $G$-invariant section such that $s\circ \pi(p)=G\cdot(p,\wt s(p))$, for every $p\in P$. Assume now that $X$ is an open subset of $\R^n$, and let $\Psi:X\x V\to E$ be a bundle map (fixing the base). We define $\na^A\Psi$ by $(\na^A_v\Psi)w:=\na^A_v(\Psi w)$, for every $x\in X$, $v\in T_xX$ and $w\in V$. (Here we think of $w$ as a constant section of $X\x V$.)
\subsection*{The Maslov index}\label{subsec:Maslov}
Let $M,\om,G,\g,\lan\cdot,\cdot\ran_\g,\mu$ and $J$ be as in Section \ref{sec:main}, $\Si=\R^2,\om_{\R^2}=\om_0,j=i$, $P\to\R^2$ a principal $G$-bundle, $p>2,\lam>1-2/p$ and $w=(u,A)\in\BB^p_\lam$. The definition of the Maslov index of $w$ is based on the following. 
\begin{prop}\label{prop:P} There exists an extension of $P$ to some smooth  $G$-bundle $\wt P\to S^2=\R^2\cup\{\infty\}$, such that $u$ extends to a continuous map from $\wt P$ to $M$. Furthermore, if $\wt P_1,\wt P_2$ are two such extensions of $P$ and $\wt u_1,\wt u_2$ the corresponding extensions of $u$, then there exists an isomorphism of continuous $G$-bundles $\Psi:\wt P_1\to \wt P_2$ such that $\wt u_1=\wt u_2\circ\Psi$. 
\end{prop}
The proof of Proposition \ref{prop:P} is postponed to the appendix (page \pageref{proof:prop:P}). We choose $\wt P$ and $\wt u$ as in Proposition \ref{prop:P}. Then $\om$ induces a fiberwise symplectic form $\wt\om$ on the continuous bundle $TM^{\wt u}=(\wt u^*TM)/G\to S^2$. 
\begin{defi}[Maslov index]\label{defi:m w} We define the Maslov index $m(w)$ to be the first Chern number of $(TM^{\wt u},\wt\om)$.
\end{defi}
It follows from Proposition \ref{prop:P} that $m(w)$ does not depend on the choice of the extension $\wt P$. Note that it only depends on the gauge equivalence class of $w$. The condition $\lam>1-2/p$ is needed for $m(w)$ to be well-defined for $w\in\BB^p_\lam$. Consider for example the case $M:=\R^2,\om:=\om_0,J:=i$ and $G:=\{e\}$. Let $p>2$ and $\lam<1-2/p$. We choose $0<\eps<1-2/p-\lam$, and a smooth map $u:\R^2\to \R^2=\C$ such that $u(z)=\sin(|z|^\eps)$, for $z\in\R^2\wo B_1$. Then $(u,0)\in\BB^p_\lam$, and $u(re^{i\phi})$ diverges, as $r\to\infty$, for every $\phi\in\R$. Therefore, we can not associate any Maslov index with $(u,0)$. 
\section{Proof of the main result}\label{sec:proof}
Let $M,\om,G,\g,\lan\cdot,\cdot\ran_\g,\mu$ and $J$ be as in Section \ref{sec:main} (assuming hypothesis (H)), $\Si:=\R^2,\om_{\R^2}:=\om_0,j:=i$ and $P\to\R^2$ a principal $G$-bundle. In the present section we always {\bf assume that $\bar n:=(\dim M)/2-\dim G>0$}. 
\subsection{Reformulation of the Fredholm theorem}\label{subsec:reform} 
Let $p>2$, $\lam\in\R$, $\BB^p_\lam$ be defined as in (\ref{eq:BB p lam}), and $w:=(u,A)\in\BB^p_\lam$ be a smooth map. We denote $\im L:=\{(x,L_x\xi)\,|\,x\in M,\,\xi\in\g\}$, and by $\Pr:TM\to TM$ the orthogonal projection onto $\im L$. $\Pr$ induces an orthogonal projection ${\Pr}^u:TM^u\to TM^u$ onto $(u^*\im L)/G$. For $\ze=(v,\al)\in TM^u\oplus\wone(\g_P)$ we write ${\Pr^u}\ze:=({\Pr^u}v,\al)$. Note that $\im L$ is in general not a subbundle of $TM$, since the dimension of $\im L_x$ may vary with $x\in M$. For $\ze\in W^{1,p}_\loc\big(TM^u\oplus\wone(\g_P)\big)$ we define $\Vert \ze\wt\Vert:=\Vert\ze\Vert_{w,p,\lam}+\Vert {\Pr}^u\ze\Vert_{p,\lam}$, where $\Vert\ze\Vert_{w,p,\lam}$ is as in (\ref{eq:ze w p lam}). Recall the definition (\ref{eq:YY w}) of $\Y_w^{p,\lam}$. We define 
\begin{eqnarray}\label{eq:X w p lam}&\XX_w:=\XX_w^{p,\lam}:=\big\{\ze\in W^{1,p}_\loc\big(TM^u\oplus\wone(\g_P)\big)\,\big|\,\Vert\ze\wt\Vert_{w,p,\lam}<\infty\big\},&\\ 
\label{eq:Y w p lam}&\YY_w:=\YY_w^{p,\lam}:=\Y_w^{p,\lam}\oplus L^p_\lam(\g_P),&\\
&\DD_w:=\DD_w^{p,\lam}:\XX_w^{p,\lam}\to\YY_w^{p,\lam},\quad \DD_w\ze:=(\D_w\ze,L_w^*\ze).&
\end{eqnarray}
Here $\D_w\ze$ is defined as in (\ref{eq:D w}). Note that the map $L_w^*:\XX_w:=\to L^p_\lam(\g_P)$ given by $L_w^*(v,\al):=L_u^*v-d_A^*\al$ is well-defined and bounded. This follows from the fact $L_x^*=L_x^*\Pr_x$ (for every $x\in M$) and compactness of $\BAR{u(P)}$. 
\begin{thm}\label{thm:Fredholm aug} Let $p>2$ and $\lam>-2/p+1$ be real numbers, and $w:=(u,A)\in\BB^p_{\lam}$ a smooth pair. Then the following statements hold. 
  \begin{enui}\item\label{thm:Fredholm X w} The normed spaces $(\XX^{p,\lam}_w,\Vert\cdot\Vert_{w,p,\lam}),\Y^{p,\lam}_w$ and $L^p_\lam(\g_P)$ are complete. 
\item\label{thm:Fredholm DD w} Assume that $-2/p+1<\lam<-2/p+2$. Then the operator $\DD^{p,\lam}_w:\XX^{p,\lam}_w\to\YY^{p,\lam}_w$ is Fredholm of real index
  \begin{equation}
    \label{eq:ind DD w}\ind\DD^{p,\lam}_w=2\bar n+2m(w).
  \end{equation}
  \end{enui}
\end{thm}
This theorem is proved in Section \ref{subsec:proof:Fredholm aug}. The proof relies on the existence of a suitable trivialization of $TM^u\oplus\wone(\g_P)$ in which the operator $\DD^{p,\lam}$ becomes standard. 
\begin{thm}\label{thm:L w * R} Let $p>2$, $\lam>1-2/p$, and $w:=(u,A)\in\BB^p_\lam\cap \BB$. Then the map $L_w^*:\XX_w\to L^p_\lam(\g_P)$ admits a bounded (linear) right inverse. 
\end{thm}
The proof of Theorem \ref{thm:L w * R} is postponed to Section \ref{subsec:proof:thm:L w * R} (see page \pageref{proof:thm:L w * R}). It is based on the existence of a bounded right inverse for the operator $d_A^*$ over a compact subset of $\R^n$ diffeomorphic to $\bar B_1$ (Proposition \ref{prop:right}) and the existence of a neighborhood $U\sub M$ of $\mu^{-1}(0)$ such that $\inf\big\{|L_x\xi|\,\big|\,x\in U,\,\xi\in\g:\,|\xi|=1\big\}>0$. We define
\begin{equation}\label{eq:M *}M^*:=\big\{x\in M\,\big|\,gx=x\then g=\one\big\}.\end{equation}
\begin{proof}[Proof of Theorem \ref{thm:Fredholm}]\setcounter{claim}{0} Let $p>2$, $\lam>1-2/p$, and $w:=(u,A)\in\BB^p_\lam$ be a smooth pair. We {\bf prove (\ref{thm:Fredholm X w})}.
\begin{claim}\label{claim:XX X} We have $\X_w:=\X_w^{p,\lam}=K:=\ker\big(L_w^*:\XX_w\to L^p_\lam(\g_P)\big)$, and the restriction of the norm $\Vert\cdot\wt\Vert_{w,p,\lam}$ to $\X_w$ is equivalent to $\Vert\cdot\Vert_{w,p,\lam}$. 
\end{claim}
\begin{proof}[Proof of Claim \ref{claim:XX X}] It suffices to prove that $\X_w\sub K$ and this inclusion is bounded. It follows from hypothesis (H) that there exists $\de>0$ such that $\mu^{-1}(\bar B_\de)\sub M^*$. We have $c:=\min\big\{|L_x\xi|\,\big|\,x\in\mu^{-1}(\bar B_\de),\,\xi\in\g:\,|\xi|=1\big\}>0$. Lemma \ref{le:si} below implies that there exists $R>0$ such that $u(P|_{\R^2\wo B_R})\sub \mu^{-1}(\bar B_\de)$. Let $\ze=(v,\al)\in\X_w$. Then $L_u^*v=d_A^*\al$, and thus, using the last assertion of Remark \ref{rmk:c} below,
\[\Vert{\Pr}^u v\Vert_{p,\lam}\leq c^{-1}\Vert L_u^*v\Vert_{p,\lam}\leq c^{-1} \Vert\na^A\al\Vert_{p,\lam}\leq c^{-1} \Vert\ze\Vert_{w,p,\lam}<\infty.\] 
Hence $\X_w\sub K$, and this inclusion is bounded. This proves Claim \ref{claim:XX X}.
\end{proof}
Part (\ref{thm:Fredholm X w}) follows from part (\ref{thm:Fredholm:X Y}) of Theorem \ref{thm:Fredholm aug} and Claim \ref{claim:XX X}. {\bf Part (\ref{thm:Fredholm:DDD})} follows from part (\ref{thm:Fredholm DD w}) of Theorem \ref{thm:Fredholm aug}, Theorem \ref{thm:L w * R} and Lemma \ref{le:X Y Z} (appendix) with $X:=\XX_w,Y:=\Y_w,Z:=L^p_\lam(\g_P)$, $D':\XX_w\to\Y_w$ given by (\ref{eq:D w}) and $T:=L_w^*$. This proves Theorem \ref{thm:Fredholm}.
\end{proof}
\subsection{Proof of Theorem \ref{thm:Fredholm aug} (augmented vertical differential)}\label{subsec:proof:Fredholm aug}
We denote by $s$ and $t$ the standard coordinates in $\R^2$. For $v\in\R^n$ we denote $\lan v\ran:=\sqrt{1+|v|^2}$. For $d\in \Z$ we define $p_d:\C\to \C$, $p_d(z):=z^d$. We equip the bundle $\wone(\g_P)$ with the (fiberwise) complex structure $J_P$ defined by $J_P\al:=-\al\,i$. Furthermore, we denote $\g^\C:=\g\otimes_\R\C$, $V:=\C^{\bar n}\oplus\g^\C\oplus\g^\C$, and for $a\in\C$ we denote by $a\cdot\oplus\id:V\to V$ the map $\big(v^1,\ldots,v^{\bar n},\al,\be\big)\mapsto \big(av^1,v^2,\ldots,v^{\bar n},\al,\be\big)$. For $x\in M$ we write $L^\C_x:\g^\C\to T_xM$ for the complex linear extension of $L_x$. We define 
\[H_x:=\ker d\mu(x)\cap(\im L_x)^\perp,\,\forall x\in M.\] 
Note that in general, the union $H$ of all the $H_x$'s is not a smooth subbundle of $TM$, since the dimension of $H_x$ may depend on $x$. However, there exists an open neighborhood $U\sub M$ of $\mu^{-1}(0)$ such that $H|_U$ is a subbundle of $TM|_U$. Let $p>2$, $\lam>-2/p+1$ and $w:=(u,A)\in\BB^p_\lam$ be a smooth pair. For $z\in\R^2$ we define $H^u_z:=\big\{G\cdot(p,v)\,\big|\,p\in\pi^{-1}(z)\sub P,\,v\in H_{u(p)}\big\}$. Consider a complex trivialization (i.e. bundle isomorphism fixing the base $\R^2$)
\[\Psi:\R^2\x V\to TM^u\oplus\wone(\g_P).\]
\begin{defi}\label{defi:triv} We call $\Psi$ \emph{good}, if the following properties are satisfied.
\begin{enui}\item {\bf (Splitting)}\label{defi:triv split} For every $z\in\R^2$ we have
  \begin{align}
    \label{eq:Psi z C bar n}\Psi_z(\C^{\bar n}\oplus\g^\C\oplus\{0\})=&\,TM^u_z\oplus\{0\},\\
\label{eq:Psi z 0} \Psi_z(\{0\}\oplus\{0\}\oplus\g^\C)=&\,\{0\}\oplus \wone(\g_P).
  \end{align}
Furthermore, there exists a number $R>0$, a smooth section $\si$ of $P\to\R^2\wo B_1$, and a point $x_\infty\in\mu^{-1}(0)$, such that the following conditions are satisfied. For every $z\in\R^2\wo B_R$ we have
  \begin{equation}\label{eq:Psi infty x H}\Psi_z(\C^{\bar n}\oplus\{0\}\oplus\{0\})=H^u_z,\end{equation}
$u\circ\si(re^{i\phi})$ converges to $x_\infty$, uniformly in $\phi\in\R$, as $r\to\infty$, $\si^*A\in L^p_\lam(\R^2\wo B_1,\g)$, and for every $z\in\R^2\wo B_R$ and $\big(\al,\be=\phi+i\psi\big)\in\g^\C\oplus\g^\C$, we have
\begin{equation}\label{eq:Psi 0 al be}\Psi_z(0,\al,\be)=\big(G\cdot\big(u\circ\si(z),L^\C_{u\circ\si(z)}(\al)\big),G\cdot\big(\si(z),\phi ds+\psi dt\big)\big).
\end{equation}
\item\label{defi:triv C} There exists a number $C>0$ such that for every $(z,\ze)\in\R^2\x V$
  \begin{equation}
    \label{eq:C}C^{-1}|\ze|\leq \big|\Psi_z(\lan z\ran^{m(w)}\cdot\oplus\id)\ze\big|\leq C|\ze|.
  \end{equation}
\item\label{defi:triv na} We have $\big|\na^A\big(\Psi(p_{m(w)}\cdot\oplus\id)\big)\big|\in L^p_\lam(\R^2\wo B_1)$.
\end{enui}
\end{defi}
\begin{prop}\label{prop:triv} If $p>2$, $\lam>-2/p+1$ and $w:=(u,A)\in\BB^p_\lam$ is smooth then there exists a good complex trivialization of $TM^u\oplus\wone(\g_P)$. 
\end{prop}
The proof of this proposition is postponed to subsection \ref{subsec:proofs} (page \pageref{proof:triv}). The next result shows that a good trivialization transforms $\DD_w$ into some standard operator. For $\al=(\al_1,\ldots,\al_n)\in (\N\disj\{0\})^n$ we denote $|\al|:=\sum_{i=1}^n\al_i$ and $\d^\al:=\d_1^{\al_1}\cdots\d_n^{\al_n}$. Let $1\leq p\leq \infty$, $n\in\N$, $k\in\N\cup\{0\}$, $\lam\in\R$, $\Om\sub\R^n$ an open subset, $W$ a real or complex vector space,  and $u:\Om\to W$ a $k$-times weakly differentiable map. We define 
\begin{eqnarray}\nn\Vert u\Vert_{L^{k,p}_\lam(\Om,W)}&:=&\sum_{|\al|\leq k}\big\Vert\langle \cdot\rangle^{\lam+|\al|}\d^\al u\big\Vert_{L^p(\Om,W)}\in[0,\infty],\\
\nn\Vert u\Vert_{W^{k,p}_\lam(\Om,W)}&:=&\sum_{|\al|\leq k}\Vert\langle\cdot\rangle^\lam\d^\al u\Vert_{L^p(\Om,W)}\in[0,\infty],\\
\label{eq:L k p}L^{k,p}_\lam(\Om,W)&:=&\big\{u\in W^{k,p}_\loc(\Om,W)\,|\,\Vert u\Vert_{L^{k,p}_\lam(\Om,W)}<\infty\big\}\\
\label{eq:W k p} W^{k,p}_\lam(\Om,W)&:=&\big\{u\in W^{k,p}_\loc(\Om,W)\,|\,\Vert u\Vert_{W^{k,p}_\lam(\Om,W)}<\infty\big\}.
\end{eqnarray}
If $(X_i,\Vert\cdot\Vert_i)$, $i=1,\ldots,k$, are normed vector spaces then we endow $X_1\oplus\cdots\oplus X_k$ with the norm $\Vert(x_1,\ldots,x_k)\Vert:=\sum_i\Vert x_i\Vert_i$. Let $d\in\Z$. If $d<0$ then we choose $\rho_0\in C^\infty(\R^2,[0,1])$ such that $\rho_0(z)=0$ for $|z|\leq1/2$ and $\rho_0(z)=1$ for $|z|\geq1$. In the case $d\geq0$ we set $\rho_0:=1$. The isomorphism of Lemma \ref{le:X d iso} (appendix) induces norms on $\XX'_{p,\lam,d}:=\C\rho_0p_d+L^{1,p}_{{\lam-1}-d}(\R^2,\C)$ and $\XX''_{p,\lam}:=\C^{\bar n-1}+L^{1,p}_{\lam-1}(\R^2,\C^{\bar n-1})$. We define
\[\XX_d:=\XX^{p,\lam}_d:=\XX'_{p,\lam,d}\oplus\XX''_{p,\lam}\oplus W^{1,p}_{\lam}(\R^2,\g^\C\oplus\g^\C),\] 
\[\YY_d:=\YY^{p,\lam}_d:=L^p_{\lam-d}(\R^2,\C)\oplus L^p_{\lam}\big(\R^2,\C^{\bar n-1}\oplus\g^\C\oplus\g^\C\big).\] 
For a complex vector space $W$ we denote by $\d^W_{\bar z}$ ($\d^W_z$) the operator $\frac12(\d_s+i\d_t)$ ($\frac12(\d_s-i\d_t)$) acting on functions from $\C$ to $W$. We denote by $\lan\cdot,\cdot\ran_\g^\C$ the hermitian inner product on $\g^\C$ (complex anti-linear in its first argument) extending $\lan\cdot,\cdot\ran_\g$. We define
\begin{eqnarray}\label{eq:F 1 F 2} &F_1:TM^u\to \wzeroone(T\C,TM^u),F_2:\wone(\g_P)\to \wtwo(\g_P)\oplus\g_P,&
\end{eqnarray}
by $F_1(v):=(ds-Jdt)v$ and  $F_2(\phi ds+\psi dt):=(\psi ds\wedge dt,\phi)$, and $F:=F_1\oplus F_2$. 
\begin{prop}[Operator in good trivialization]\label{prop:X X w} Let $2<p<\infty$, $\lam>-2/p+1$, $w:=(u,A)\in \BB^p_{\lam}$ be smooth and $\Psi:\R^2\x V\to TM^u\oplus\wone(\g_P)$ a good trivialization. Then the following statements hold.
  \begin{enui}\item \label{prop:X X w iso} The following maps are well-defined isomorphisms of normed spaces:
  \begin{equation} \label{eq:X Y Psi}\XX_{m(w)}\ni\ze\mapsto \Psi\ze\in \XX_w,\quad\YY_{m(w)}\ni\ze\mapsto F\Psi\ze\in \YY_w
  \end{equation}
\item\label{prop:X X w Fredholm} There exists a positive $\C$-linear map $S_\infty:\g^\C\to \g^\C$ (i.e. $\lan S_\infty v,v\ran_\g^\C>0$ for every $0\neq v\in\g^\C$) such that the following operator is compact:
\begin{equation}\label{eq:F Psi DD}S:=(F\Psi)^{-1}\DD_w \Psi-\d^{\C^{\bar n}}_{\bar z}\oplus\left(
  \begin{array}{cc}
\d^{\g^\C}_{\bar z}&\id/2\\
S_\infty&2\d^{\g^\C}_z
  \end{array}
\right):\XX_{m(w)}\to\YY_{m(w)}
\end{equation}
\end{enui}
\end{prop}
The proof of Proposition \ref{prop:X X w} is postponed to subsection \ref{subsec:proofs} (page \pageref{proof:X X w}). It is based on some inequalities and compactness properties for weighted Sobolev spaces (Proposition \ref{prop:lam d Morrey}) and a Hardy-type inequality (Proposition \ref{prop:Hardy}). 
\begin{proof}[Proof of Theorem \ref{thm:Fredholm aug}] \label{proof:Fredholm aug} Let $p>2$, $\lam>-2/p+1$, and let $w:=(u,A)\in \BB^p_\lam$ be a smooth pair. The space (\ref{eq:L k p}) is complete, see \cite{Lockhart PhD}. By Proposition \ref{prop:lam d Morrey}(\ref{prop:lam W}) (appendix) the same holds for the space (\ref{eq:W k p}). Combining this with Propositions \ref{prop:triv} and \ref{prop:X X w}(\ref{prop:X X w iso}), {\bf part (\ref{thm:Fredholm X w})} follows. {\bf Part (\ref{thm:Fredholm DD w})} follows from Propositions \ref{prop:triv} and \ref{prop:X X w}(\ref{prop:X X w Fredholm}), Corollary \ref{cor:d L L} and Proposition \ref{prop:d A B d} (appendix). This proves Theorem \ref{thm:Fredholm aug}.\end{proof}
\begin{rmk}\label{rmk:Fredholm aug} An alternative approach to prove Theorem \ref{thm:Fredholm aug} is to switch to ``logarithmic'' coordinates $\tau+i\phi$ (defined by $e^{\tau+i\phi}=s+it\in\R^2\wo\{0\}$). In these coordinates and a suitable trivialization the operator $\DD^{p,\lam}$ is of the form $\d_\tau+A(\tau)$. Hence one can try to apply the results of \cite{RoSa}. However, this is not possible, since $A(\tau)$ contains the operator $v\mapsto e^{2\tau} d\mu(u)v\,d\tau\wedge d\phi$, which diverges for $\tau\to\infty$. 
\end{rmk}
\subsection{Proofs of the results of subsection \ref{subsec:proof:Fredholm aug}}\label{subsec:proofs}
\begin{proof}[Proof of Proposition \ref{prop:triv} ]\label{proof:triv}\setcounter{claim}{0} Let $p,\lam$ and $w$ be as in the hypothesis. We choose a section $\si$ of $P|_{\R^2\wo B_1}$ and a point $x_\infty\in\mu^{-1}(0)$ as in Lemma \ref{le:si}. 
\begin{claim}\label{claim:U triv} There exists an open $G$-invariant neighborhood $U\sub M$ of $x_\infty$ such that $H|_U$ is a smooth subbundle of $TM$ with the following property. There exists a smooth complex trivialization $\Psi^U:U\x\C^{\bar n}\to H|_U$ satisfying $\Psi^U_{gx}v_0=g\Psi^U_xv_0:=g\Psi^U(x,v_0)$, for every $g\in G$, $x\in U$ and $v_0\in \C^{\bar n}$. 
\end{claim}
\begin{proof}[Proof of Claim \ref{claim:U triv}] By hypothesis (H) we have $x_\infty\in M^*$, where $M^*$ is defined as in (\ref{eq:M *}). We choose a $G$-invariant neighborhood $U_0\sub M^*$ of $x_\infty$ so small that $\ker d\mu(x)$ and $(\im L_x)^\perp$ intersect transversely, for every $x\in U_0$. Then $H|_{U_0}$ is a smooth subbundle of $TM|_{U_0}$. Furthermore, by the local slice theorem there exists a pair $(U,N)$, where $U\sub U_0$ is a $G$-invariant neighborhood of $x_\infty$ and $N\sub U$ is a submanifold of dimension $\dim M-\dim G$ that intersects $Gx$ transversely in exactly one point, for every $x\in U$. We choose a complex trivialization of $H|_N$ and extend it in a $G$-equivariant way, to obtain a trivialization $\Psi^U$ of $H|_U$. This proves Claim \ref{claim:U triv}.
\end{proof}
We choose $U$ and $\Psi^U$ as in Claim \ref{claim:U triv}. It follows from Lemma \ref{le:si} that there exists $R>1$ such that $u(p)\in U$, for $p\in\pi^{-1}(z)\sub P$, if $z\in\R^2\wo B_R$. We define $\wt\Psi^\infty:(\R^2\wo B_R)\x(\C^{\bar n}\oplus\g^\C)\to TM^u=(u^*TM)/G$ by 
\[\wt\Psi^\infty_z(v_0,\al)=G\cdot\big(u\circ\si(z),\Psi^U_{u\circ\si(z)}(z^{-m(w)}\cdot\oplus\id)v_0+L^\C_{u\circ\si(z)}\al\big),\]
This is a smooth complex trivialization of $TM^u|_{\R^2\wo B_R}$. 
\begin{claim}\label{claim:wt Psi 1} $\wt\Psi^\infty|_{\C\wo B_{R+1}}$ extends to a smooth complex trivialization of $TM^u$. 
\end{claim}
We define $f:\C\wo\{0\}\to S^1$ by $f(z):=z/|z|$. 
\begin{proof}[Proof of Claim \ref{claim:wt Psi 1}] We choose a complex trivialization  $\Psi^0:\bar B_R\x (\C^{\bar n}\oplus\g^\C)\to TM^u|_{\bar B_R}$. We define $\Phi:S^1_R:=\{z\in\C\,|\,|z|=R\}\to \Aut(\C^{\bar n}\oplus\g^\C)$ by 
\[\Phi_z(v_0,\al):=(\Psi^0_z)^{-1}\big(G\cdot\big(u\circ\si(z),\Psi^U_{u\circ\si(z)}v_0+L^\C_{u\circ\si(z)}\al\big)\big).\] 
For a continuous map $x:S^1_R\to S^1$ we denote by $\deg(x)$ its degree.
\begin{claim}\label{claim:d m(w)} The index $m(w)$ equals $\deg(f\circ\det\circ\Phi)$. 
\end{claim}
\begin{proof}[Proof of Claim \ref{claim:d m(w)}] We define $\wt P$ to be the quotient of $P\disj \big((S^2\wo\{0\})\x G\big)$ under the equivalence relation generated by $p\sim (z,g)$, where $g\in G$ is determined by $\si(z)g=p$, for $p\in \pi^{-1}(z)\sub P$, $z\in \C\wo\{0\}$. Furthermore, we define $\wt u:\wt P\to M$ by $\wt u([p]):=u(p)$, for $p\in P$, and $\wt u([\infty,g]):=g^{-1}x_\infty$, for $g\in G$. The statement of Lemma \ref{le:si} implies that this map is continuous and extends $u$. The (fiberwise linear) complex structure $u^*J$ on $u^*TM$ descends to a complex structure $\wt J$ on $TM^{\wt u}=(\wt u^*TM)/G\to S^2$. By definition, we have $m(w)=c_1\big(TM^{\wt u},\wt J\big)$. We define the local trivialization $\Psi^\infty:(S^2\wo B_R)\x(\C^{\bar n}\oplus\g^\C)\to TM^{\wt u}$ by 
\[\Psi^\infty_z(v_0,\al):=\left\{\begin{array}{ll}G\cdot\big([u\circ\si(z)],\Psi^U_{u\circ\si(z)}v_0+L^\C_{u\circ\si(z)}\al\big),&\textrm{if }z\in \R^2\wo B_R,\\
 G\cdot\big([\infty,\one],\Psi^U_{x_\infty}v_0+L^\C_{x_\infty}\al\big),&\textrm{if }z=\infty.
\end{array}\right.\]
Then $\Phi_z=(\Psi^0_z)^{-1}\Psi^\infty_z$, for $z\in S^1_R$, and therefore $\Phi$ is the transition map between $\Psi^0$ and $\Psi^\infty$. Claim \ref{claim:d m(w)} follows from this.
\end{proof}
By Claim \ref{claim:d m(w)} and Lemma \ref{le:Aut} in the appendix the maps $\Phi$ and $S^1_R\ni z\mapsto (z^{m(w)}\cdot\oplus\id)\in\Aut(\C^{\bar n}\oplus\g^\C)$ are homotopic. Hence there exists a continuous map $h:\bar B_R\wo B_1\to \Aut(\C^{\bar n}\oplus\g^\C)$ such that $h_z:=h(z)=(z^{m(w)}\cdot\oplus\id)$, if $z\in S^1_1$, and $h_z=\Phi(z)$, if $z\in S^1_R$. We define $\wt\Psi:\R^2\x(\C^{\bar n}\oplus\g^\C)\to TM^u$ by $\wt\Psi:=\wt\Psi^\infty$ on $\R^2\wo B_R$, and $\wt\Psi_z(v_0,\al):=\Psi^0_zh_z(z^{-m(w)}\cdot\oplus\id)(v_0,\al)$, for $z\in B_R$, $(v_0,\al)\in\C^{\bar n}\oplus\g^\C$. Smoothing $\wt\Psi$ out on the ball $B_{R+1}$, we obtain the required extension of $\wt\Psi^\infty|_{\C\wo B_{R+1}}$. This proves Claim \ref{claim:wt Psi 1}.
\end{proof}
We define $\hhat\Psi^\infty:(\R^2\wo B_R)\x\g^\C\to \wone\big((P|_{\R^2\wo B_R}\x \g)/G\big)$ by 
\[\hhat\Psi^\infty_z(\phi+i\psi):=G\cdot\big(\si(z),\phi ds + \psi dt\big).\] 
\begin{claim}\label{claim:hhat Psi 1} $\hhat\Psi^\infty|_{\R^2\wo B_{R+1}}$ extends to a smooth complex trivialization of the bundle $\wone(\g_P)\to \C$. 
\end{claim}
\begin{proof}[Proof of Claim \ref{claim:hhat Psi 1}] We denote by $\Ad$ and $\Ad^\C$ the adjoint representations of $G$ on $\g$ and $\g^\C$ respectively. For every $g\in G$ we have $\det(\Ad^\C_g)=\det(\Ad_g)\in\R$. We choose a continuous section $\wt\si$ of the restriction $P|_{\bar B_R}$. We define $g:S^1_R\to G$ to be the unique map such that $\si(z)=\wt\si(z) g(z)$, for every $z\in S^1_R$. It follows that $f\circ\det(\Ad^\C_g)\const \pm 1$. Therefore, $\deg\big(S^1_R\ni z\mapsto f\circ\det(\Ad^\C_{g(z)})\big)=0$. Hence Lemma \ref{le:Aut} (appendix) implies that there exists a continuous map $\Phi:\bar B_R\to \Aut(\g^\C)$ satisfying $\Phi_z:=\Phi(z)=\Ad^\C_{g(z)}$, for every $z\in S^1_R$. We define $\hhat\Psi:\R^2\x\g^\C\to \wone(\g_P)$ by $\hhat\Psi:=\hhat\Psi^\infty$ on $\R^2\wo B_R$ and by $\hhat\Psi_z\al:=G\cdot\big(\wt\si,\phi'ds + \psi' dt\big)$, where $\phi'+ i\psi':=\Phi_z\al$, for $z\in B_R$, $\al\in \g^\C$. Smoothing $\hhat\Psi$ out on the ball $B_{R+1}$, we obtain the required extension of $\hhat\Psi^\infty|_{\R^2\wo B_{R+1}}$. This proves Claim \ref{claim:hhat Psi 1}.
\end{proof}
We choose extensions $\wt\Psi$ and $\hhat\Psi$ of $\wt\Psi^\infty$ and $\hhat\Psi^\infty$ as in Claims \ref{claim:wt Psi 1} and \ref{claim:hhat Psi 1}. 
\begin{claim}\label{claim:Psi good} $\Psi:=\wt\Psi\oplus\hhat\Psi$ is a good complex trivialization of $TM^u\oplus\wone(\g_P)$.
\end{claim} 
\begin{pf}[Proof of Claim \ref{claim:Psi good}] Condition {\bf (\ref{defi:triv split})} of Definition \ref{defi:triv} follows from the construction of $\Psi$. To prove {\bf (\ref{defi:triv C})}, note that for $z\in\R^2\wo B_{R+1}$ and $(v_0,\al,\be)\in V$, we have
\begin{equation}\label{eq:Psi z z}\big|\Psi_z(z^{m(w)}\cdot\oplus\id)(v_0,\al,\be)\big|^2=\big|\Psi^U_{u\circ\si(z)}v_0\big|^2+\big|L^\C_{u\circ\si(z)}\al\big|^2+|\be|^2.\end{equation}
Here we used the fact $H_x=(\im L^\C_x)^\perp$, for every $x\in M$. By our choice of $U$, $H|_U\sub TM|_U$ is a smooth subbundle of rank $\dim M-2\dim G$. It follows that $\im L^\C|_U=H^\perp|_U$ is a smooth subbundle of $TM|_U$ of rank $2\dim G$. Hence $L^\C_x:\g^\C\to T_xM$ is injective, for every $x\in U$. Since by assumption $\BAR{u(P)}\sub M$ is compact, the same holds for the set $\BAR{u(P|_{\R^2\wo B_{R+1}})}\sub\BAR{u(P)}$. It follows that there exists a constant $C>0$ such that
\[C^{-1}|v_0|\leq\big|\Psi^U_{u\circ\si(z)}v_0\big|\leq C|v_0|,\quad C^{-1}|\al|\leq \big|L^\C_{u\circ\si(z)}\al\big|\leq C|\al|,\]
for every $z\in \R^2\wo B_{R+1}$, $v_0\in\C^{\bar n}$ and $\al\in\g^\C$. Combining this with equality (\ref{eq:Psi z z}), condition (\ref{defi:triv C}) follows. 

We check condition {\bf (\ref{defi:triv na})}. Let $\ze:=\big(v_0,\al,\be=\phi+i\psi\big)\in V$, $z\in \R^2\wo B_{R+1}$ and $v\in T_z\R^2$. We choose a point $p\in \pi^{-1}(z)\sub P$ and a vector $\wt v\in T_pP$ such that $\pi_*\wt v=v$. Then
\begin{equation}\label{eq:na A v wt Psi}\na^A_v\big(\wt\Psi(p_{m(w)}\cdot\oplus\id)(v_0,\al)\big)=G\cdot\big(u(p),\wt\na^A_{\wt v}(\Psi^U_uv_0+L^\C_u\al)\big).\end{equation}
Furthermore, for every smooth vector field $X$ on $U$ we have
\begin{equation}\label{eq:wt na A wt v X}\wt\na^A_{\wt v}X=(u^*\na)_{\wt v-pA\wt v}X=\na_{d_Au\cdot\wt v}X.\end{equation}
We define $C$ to be the maximum of $\big|\na_{v'}(\Psi^U_xv''+L^\C_x\al)\big|$ over all $v'\in T_xM$, $x\in \BAR{u(P|_{\R^2\wo B_{R+1}})}$ and $(v'',\al)\in \C^{\bar n}\oplus\g^\C$ such that $|v'|\leq1$, $|(v'',\al)|\leq1$. Furthermore, we define $C':=\Vert d_Au\Vert_{L^p_\lam(\R^2\wo B_{R+1})}$. By (\ref{eq:na A v wt Psi}) and (\ref{eq:wt na A wt v X}) with $X(x):=\Psi^U_xv_0+L^\C_x\al$, we have
\begin{equation}\label{eq:Vert na A v wt Psi}\big\Vert\na^A_v\big(\wt\Psi(p_{m(w)}\cdot\oplus\id)(v_0,\al)\big)\big\Vert_{L^p_\lam(\R^2\wo B_{R+1})}\leq CC'|v||(v_0,\al)|,
\end{equation}
We now define $\wt\phi,\wt\psi:P\to g$ to be the unique equivariant maps such that $\wt\phi\circ\si\const\phi$, $\wt\psi\circ\si\const\psi$. We have $d_A\wt\phi\, \si_*v=[(\si^*A)v,\phi]$, and similarly for $\wt\psi$. Since $\wt\na^A_{\si_*v}(\wt\phi ds+\wt\psi dt)=(d_A\wt\phi\si_*v)ds+(d_A\wt\psi\si_*v)dt$, it follows that
\[\big|\na^A_v\big(\hhat\Psi(\phi ds+\psi dt)\big)\big|=\big|G\cdot\big(\si(z),\wt\na^A_{\si_*v}(\wt\phi ds+\wt\psi dt)\big)\big|=\big|\big[(\si^*A)v,\be\big]\big|.\]
Combining this with (\ref{eq:Vert na A v wt Psi}) and the facts $d_Au\in L^p_\lam(\R^2\wo B_{R+1})$ and $\Vert\si^*A\Vert_{p,\lam}<\infty$, condition (\ref{defi:triv na}) follows. This proves Claim \ref{claim:Psi good} and concludes the proof of Proposition \ref{prop:triv}.
\end{pf}
\end{proof}
\begin{proof}[Proof of Proposition \ref{prop:X X w}]\label{proof:X X w} \setcounter{claim}{0}
Let $p,\lam,w=(u,A)$ and $\Psi$ be as in the hypothesis. We choose $\rho_0\in C^\infty(\R^2,[0,1])$ such that $\rho_0(z)=0$ for $z\in B_{1/2}$ and $\rho_0(z)=1$ for $z\in\R^2\wo B_1$. We fix $R\geq1$, $\si$ and $x_\infty$ as in Definition \ref{defi:triv}(\ref{defi:triv split}). We abbreviate $d:=m(w)$. 

{\bf We prove (\ref{prop:X X w iso}).} For every $\ze\in W^{1,1}_\loc(\R^2,V)$ Leibnitz' rule implies that 
\begin{equation}\label{eq:na Phi Psi F ze}\na^A(\Psi\ze)=\big(\na^A\big(\Psi(p_d\cdot\oplus\id)\big)\big)(p_{-d}\cdot\oplus\id)\ze+\Psi(p_d\cdot\oplus\id)D\big((p_{-d}\cdot\oplus\id)\ze\big).
\end{equation}
\begin{claim}\label{claim:F X bdd} The first map in (\ref{eq:X Y Psi}) is well-defined and bounded.
\end{claim}
\begin{proof}[Proof of Claim \ref{claim:F X bdd}] Proposition \ref{prop:lam d Morrey}(\ref{prop:lam Morrey}) below and the fact $\lam>-2/p+1$ imply that there exists a constant $C_1$ such that 
\begin{equation}\label{eq:lan ran - d id ze}\big\Vert(\lan\cdot\ran^{-d}\cdot\oplus\id)\ze\big\Vert_\infty\leq C_1\Vert\ze\Vert_{\XX_d},\quad\forall \ze\in\XX_d\,.\end{equation}
We choose a constant $C_2:=C$ as in part (\ref{defi:triv C}) of Definition \ref{defi:triv}. Then by (\ref{eq:C}) and (\ref{eq:lan ran - d id ze}), we have 
\begin{equation}\label{eq:Psi ze infty}\Vert\Psi\ze\Vert_\infty\leq C_1C_2\Vert\ze\Vert_{\XX_d}\quad\forall \ze\in\XX_d\,.\end{equation}
It follows from (\ref{eq:Psi infty x H}) and (\ref{eq:Psi 0 al be}), the definition $H_x:=\ker d\mu(x)\cap \im L_x^\perp$ and the compactness of $\BAR{u(P)}$ that there exists $C_3\in\R$ such that, for every $\ze\in\XX_d$,
\begin{equation}\label{eq:d mu u v'}\big\Vert \,|d\mu(u)v'|+|{\Pr}^uv'|+|\al'|\,\big\Vert_{p,\lam}\leq C_3\Vert\ze\Vert_{\XX_d},
\end{equation}
where $(v',\al'):=\Psi\ze$. For $r>0$ we denote $B_r^C:=\R^2\wo B_r$ and $\Vert\cdot\Vert_{p,\lam;r}:=\Vert\cdot\Vert_{L^p_\lam(B_r^C)}$. We define $C_4:=\max\big\{\big\Vert\na^A\big(\Psi(p_d\cdot\oplus\id)\big)\big\Vert_{p,\lam;1},C_2\big\}$. By condition (\ref{defi:triv na}) of Definition \ref{defi:triv} we have $C_4<\infty$. Let $\ze\in\XX_d$. Then by (\ref{eq:na Phi Psi F ze}) we have
\begin{equation}\label{eq:na A Psi ze p lam}\Vert\na^A(\Psi\ze)\Vert_{p,\lam;1}\leq C_4\big(\Vert(p_{-d}\cdot\oplus\id)\ze\Vert_{L^\infty(B_1^C)}+\big\Vert D\big((p_{-d}\cdot\oplus\id)\ze\big)\big\Vert_{p,\lam;1}\big).
\end{equation}
We define $C_5:=\max\big\{-d2^{(-d+3)/2},2\big\}$. Then $\big\Vert D\big((p_{-d}\cdot\oplus\id)\ze\big)\big\Vert_{p,\lam;1}\leq C_5\Vert\ze\Vert_{\XX_d}$ by Proposition \ref{prop:lam d Morrey}(\ref{prop:lam d iso}). Combining this with (\ref{eq:na A Psi ze p lam}) and (\ref{eq:lan ran - d id ze}), we get
\begin{equation}\label{eq:na A Psi ze} \Vert\na^A(\Psi\ze)\Vert_{p,\lam;1}\leq C_4\big(2^{\frac{|d|}2}C_1+C_5\big)\Vert\ze\Vert_{\XX_d}.\end{equation}
By a direct calculation there exists a constant $C_6$ such that $\Vert\na^A(\Psi\ze)\Vert_{L^p(B_1)}\leq C_6\Vert\ze\Vert_{\XX_d}$, for every $\ze\in\XX_d$. Claim \ref{claim:F X bdd} follows from this and (\ref{eq:Psi ze infty},\ref{eq:d mu u v'},\ref{eq:na A Psi ze}).
\end{proof} 
\begin{claim}\label{claim:F -1 X bdd} The map $\XX_w\ni\ze'\mapsto \Psi^{-1}\ze'\in \XX_d$ is well-defined and bounded. 
\end{claim}
\begin{proof}[Proof of Claim \ref{claim:F -1 X bdd}] We choose a neighborhood $U\sub M$ of $\mu^{-1}(0)$ as in Lemma \ref{le:U} (appendix), and define $c$ as in (\ref{eq:c inf}), and $C_1:=\max\{c^{-1},1\}$. Since $u\circ\si(re^{i\phi})$ converges to $x_\infty$, uniformly in $\phi\in\R$, as $r\to\infty$, there exists $R'\geq R$ such that $u(p)\in U$, for every $p\in \pi^{-1}(B_{R'}^C)\sub P$. Then (\ref{eq:Psi infty x H},\ref{eq:Psi 0 al be}) and (\ref{eq:c inf}) imply that 
\begin{equation}\label{eq:Vert al be Vert}\big\Vert(\al,\be)\big\Vert_{p,\lam;R'}\leq C_1\big\Vert\Psi(0,\al,\be)\big\Vert_{p,\lam;R'}\leq C_1\Vert\ze'\Vert_w,\end{equation}
where $(v_0,\al,\be):=\Psi^{-1}\ze'$, for every $\ze'\in\XX_d$. 
\begin{claim}\label{claim:d rho 0 p} There exists a constant $C_2$ such that for every $\ze'\in\XX_w$, we have
\begin{equation}
\label{eq:d rho 0 p}\big\Vert D\big((\rho_0p_{-d}\cdot\oplus\id)\Psi^{-1}\ze'\big)\big\Vert_{L^p_{\lam}(\R^2)}\leq C_2\Vert\ze'\Vert_w.
\end{equation}
\end{claim}
\begin{proof}[Proof of Claim \ref{claim:d rho 0 p}] It follows from equality (\ref{eq:na Phi Psi F ze}) and conditions (\ref{defi:triv C}) and (\ref{defi:triv na}) of Definition \ref{defi:triv} that there exist constants $C$ and $C'$ such that
\begin{eqnarray}\nn &\big\Vert D\big((p_{-d}\cdot\oplus\id)\Psi^{-1}\ze'\big)\big\Vert_{p,\lam;1}&\\
\label{eq:p - d id F}&\leq C\big(\Vert\na^A\ze'\Vert_{p,\lam}+\big\Vert\na^A\big(\Psi(p_d\cdot\oplus\id)\big)\big\Vert_{p,\lam}\Vert\ze'\Vert_\infty\big)\leq C'\Vert\ze'\Vert_w,&
\end{eqnarray}
for every $\ze'\in\XX_w$. On the other hand, Leibnitz' rule implies 
\[D(\Psi^{-1}\ze')=\Psi^{-1}\big(\na^A\ze'-(\na^A\Psi)\Psi^{-1}\ze'\big).\]
Hence by a short calculation, using Leibnitz' rule again, it follows that there exists a constant $C''$ such that
\[\big\Vert D\big((\rho_0p_{-d}\cdot\oplus\id)\Psi^{-1}\ze'\big)\big\Vert_{L^p(B_1)}\leq C''\Vert\ze'\Vert_w,\]
for every $\ze'\in \XX_w$. Combining this with (\ref{eq:p - d id F}), Claim \ref{claim:d rho 0 p} follows. \end{proof}
Let $\ze'\in\XX_w$. We denote $\wt\ze:=(\wt v_0,\wt\al,\wt\be):=(\rho_0p_{-d}\cdot\oplus\id)\Psi^{-1}\ze'$. By inequality (\ref{eq:d rho 0 p}) the hypotheses of Proposition \ref{prop:Hardy} with $n:=2$ and $\lam$ replaced by $\lam-1$ are satisfied. It follows that there exists $\ze_\infty:=\big(v_\infty,\al_\infty,\be_\infty\big)\in V=\C^{\bar n}\oplus\g^\C\oplus\g^\C$, such that $\wt\ze(re^{i\phi})\to\ze_\infty$, uniformly in $\phi\in\R$, as $r\to\infty$, and 
\begin{equation}\label{eq:wt ze ze infty}\Vert(\wt\ze-\ze_\infty)|\cdot|^{\lam-1}\big\Vert_{L^p(\R^2)}\leq (\dim M+2\dim G)p/(\lam+2/p)\big\Vert D\wt\ze|\cdot|^{\lam}\big\Vert_{L^p(\R^2)}
\,.\end{equation}
Since $\lam>-2/p+1$, we have $\int_{B_{R'}^C}\lan\cdot\ran^{p\lam}=\infty$. Hence the convergence $(\wt\al,\wt\be)\to (\al_\infty,\be_\infty)$ and the estimate (\ref{eq:Vert al be Vert}) imply that $(\al_\infty,\be_\infty)=(0,0)$. We choose a constant $C>0$ as in part (\ref{defi:triv C}) of Definition \ref{defi:triv}. The convergence $\wt v_0\to v_\infty$ and the first inequality in (\ref{eq:C}) imply that 
\begin{equation}\label{eq:v infty 1}|v_\infty|\leq\Vert\wt v_0\Vert_\infty\leq 2^{\frac{|d|}2}C\Vert\ze'\Vert_\infty  
\,.\end{equation}
We define $\big(v^1,\ldots,v^{\bar n},\al,\be\big):=\Psi^{-1}\ze'-\big(\rho_0p_dv_\infty^1,v_\infty^2,\ldots,v_\infty^{\bar n},0,0\big)$. Proposition \ref{prop:lam d Morrey}(\ref{prop:lam d iso}) in the appendix and inequalities (\ref{eq:wt ze ze infty}) and (\ref{eq:d rho 0 p}) imply that there exists a constant $C_6$ (depending on $p,\lam,d$ and $\Psi$, but not on $\ze'$) such that 
\begin{equation}\label{eq:v 1 v 2 v bar n}\Vert v^1\Vert_{L^{1,p}_{{\lam-1}-d}(B_1^C)}+\big\Vert\big(v^2,\ldots,v^{\bar n},\al,\be\big)\big\Vert_{L^{1,p}_{\lam-1}(B_1^C)}\leq C_6\Vert\ze'\Vert_w\,.
\end{equation}
Finally, by a straight-forward argument, there exists a constant $C_7$ (independent of $\ze'$) such that $\Vert\Psi^{-1}\ze'\Vert_{W^{1,p}(B_{R'})}\leq C_7\Vert\ze'\Vert_w$. Combining this with (\ref{eq:Vert al be Vert},\ref{eq:v infty 1},\ref{eq:v 1 v 2 v bar n}), Claim \ref{claim:F -1 X bdd} follows.\end{proof}
Claims \ref{claim:F X bdd} and \ref{claim:F -1 X bdd} imply that the first map in (\ref{eq:X Y Psi}) is an isomorphism (of normed vector spaces). It follows from condition (\ref{defi:triv C}) of Definition \ref{defi:triv} that the second map in (\ref{eq:X Y Psi}) is an isomorphism. This completes the proof of {\bf (\ref{prop:X X w iso})}.

{\bf We prove statement (\ref{prop:X X w Fredholm}).} Recall that we have chosen $R>0,\si$ and $x_\infty$ as in Definition \ref{defi:triv}(\ref{defi:triv split}). We define $S_\infty:\g^\C\to\g^\C$ to be the complex linear extension of $L_{x_\infty}^*L_{x_\infty}:\g\to\g$. By our hypothesis (H) the Lie group $G$ acts freely on $\mu^{-1}(0)$. It follows that $L_{x_\infty}$ is injective. Therefore $S_\infty$ is positive with respect to $\lan\cdot,\cdot\ran_\g^\C$. By (\ref{eq:Psi z C bar n}) and (\ref{eq:Psi z 0}) there exist complex trivializations 
\[\Psi_1:\R^2\x(\C^{\bar n}\oplus\g^\C)\to  TM^u,\qquad \Psi_2:\R^2\x\g^\C\to \wone(\g_P),\] 
such that $\Psi=\Psi_1\oplus\Psi_2$. We denote by $\iota:\g^\C\to \C^{\bar n}\oplus\g^\C$ ($\pr:\C^{\bar n}\oplus\g^\C\to \g^\C$) the inclusion as (the projection onto) the second factor. We define 
\[\XX_d^1:=L^{1,p}_{{\lam-1}-d}(\R^2,\C)\oplus L^{1,p}_{\lam-1}(\R^2,\C^{\bar n-1})\oplus W^{1,p}_{\lam}(\R^2,\g^\C),\quad\XX_d^2:=W^{1,p}_{\lam}(\R^2,\g^\C),\]
\[\XX_d':=\XX_d^1\oplus\XX_d^2,\quad \XX_d^0:=\C\rho_0p_d\oplus\C^{\bar n-1}\oplus\{(0,0)\}\sub\XX_d,\]
\[\YY_d^1:=L^p_{\lam-d}(\R^2,\C)\oplus L^p_{\lam}(\R^2,\C^{\bar n-1}\oplus\g^\C),\quad \YY_d^2:=L^p_\lam(\R^2,\g^\C).\]
Note that $\XX_d=\XX_d^0+\XX_d'$ and $\YY_d=\YY_d^1\oplus\YY_d^2$. We define $S:\XX_d\to\YY_d$ as in (\ref{eq:F Psi DD}). Since $\XX_d^0$ is finite dimensional, $S|_{\XX_d^0}$ is compact. Hence it suffices to prove that $S|_{\XX_d'}$ is compact. To see this, we denote 
\[Q:=\left(\begin{array}{c}ds\wedge dt\,d\mu(u)\\
L_u^*
\end{array}\right),\quad T:=\left(\begin{array}{c}d_A\\
 -d_A^*
\end{array}\right),\]
and we define $S^i_{\phantom{i}j}:\XX_d^j\to\YY_d^i$ (for $i,j=1,2$) and $\wt S^1_{\phantom{1}1}:\XX_d^1\to\YY_d^1$ by 
\[S^1_{\phantom{1}1}v:=(F_1\Psi_1)^{-1}((\na^A\Psi_1)v)^{0,1},\,\wt S^1_{\phantom{1}1}v:=-(F_1\Psi_1)^{-1}\big(J(\na_{\Psi_1v}J)(d_Au)^{1,0}/2\big),\]
\[ S^1_{\phantom{1}2}\al:=(F_1\Psi_1)^{-1}(L_u\Psi_2\al)^{0,1}-\iota\al/2,\, S^2_{\phantom{2}1}v:=\big((F_2\Psi_2)^{-1}Q\Psi_1- S_\infty\pr\big)v,\] 
\[S^2_{\phantom{2}2}\al:=(F_2\Psi_2)^{-1}(T\Psi_2)\al.\]
Here $F_1$ and $F_2$ are as in (\ref{eq:F 1 F 2}) and $(T\Psi_2)\al:=T(\Psi_2\al)$, for $\al\in\g^\C$ (viewed as a constant section of $\R^2\x\g^\C$). A direct calculation shows that $S(v,\al)=\big(S^1_{\phantom{1}1}v+\wt S^1_{\phantom{1}1}v+S^1_{\phantom{1}2}\al,S^2_{\phantom{2}1}v+S^2_{\phantom{2}2}\al\big)$. For a subset $X\sub \R^2$ we denote by $\chi_X:\R^2\to\{0,1\}$ its characteristic function. It follows that $\chi_{B_R}S|_{\XX_d'}$ is of 0-th order. Since it vanishes outside $B_R$, it follows that this map is compact. 
\begin{claim}\label{claim:S i j compact} The operators $\chi_{B_R^C}S^i_{\phantom{i}j}$, $i,j=1,2$, and $\chi_{B_R^C}\wt S^1_1$ are compact. 
\end{claim}
\begin{pf}[Proof of Claim \ref{claim:S i j compact}] To see that the map $\chi_{B_R^C}S^1_{\phantom{1}1}$ is compact, note that Leibnitz' rule and holomorphicity of $p_d$ imply that 
\begin{equation}\nn\label{eq:na F 1 }(\na^A\Psi_1)^{0,1}=\big(\na^A\big(\Psi_1(p_d\cdot\oplus\id)\big)\big)^{0,1}(p_{-d}\cdot\oplus\id),\quad \textrm{on }\C\wo\{0\}.\end{equation} 
Since $\lam>1-2/p$, assertions (\ref{prop:lam d iso}) and (\ref{prop:lam Morrey}) of Proposition \ref{prop:lam d Morrey} imply that the map $(\rho_0p_{-d}\cdot\oplus\id):\XX_d^1\to C_b(\R^2,\C^{\bar n}\oplus\g^\C)$ is well-defined and compact. By condition (\ref{defi:triv na}) of Definition \ref{defi:triv}, the map
\begin{equation}
\nn  \label{eq:na F 1 L}\chi_{B_R^C}\big(\na^A\big(\Psi_1(p_d\cdot\oplus\id)\big)\big)^{0,1}:C_b(\R^2,\C^{\bar n}\oplus\g^\C)\to L^p_{\lam}\big(\wzeroone(\R^2,TM^u)\big)
\end{equation}
is bounded. Condition (\ref{defi:triv C}) of Definition \ref{defi:triv} implies boundedness of the map $(F_1\Psi_1)^{-1}:L^p_{\lam}\big(\wzeroone(\R^2,TM^u)\big)\to\YY_d^1$. Compactness of $\chi_{B_R^C}S^1_{\phantom{1}1}$ follows. 

By the definition of $\BB^p_\lam$, we have $|d_Au|\in L^p_{\lam}(\R^2)$. This together with Proposition \ref{prop:lam d Morrey}(\ref{prop:lam d iso}) and (\ref{prop:lam Morrey}) and  Definition \ref{defi:triv}(\ref{defi:triv C}) implies that the map $\chi_{B_R^C}\wt S^1_{\phantom{1}1}$ is compact. Furthermore, it follows from Definition \ref{defi:triv}(\ref{defi:triv split}) that $\chi_{B_R^C}S^1_{\phantom{1}2}=0$.

To see that $\chi_{B_R^C}S^2_{\phantom{2}1}$ is compact, we define $f:B_R^C\to \End(\g^\C)$ by setting $f(z):\g^\C\to\g^\C$ to be the complex linear extension of the map $L_{u\circ\si(z)}^*L_{u\circ\si(z)}-L_{x_\infty}^*L_{x_\infty}:\g\to\g$. Since $u\circ\si(re^{i\phi})$ converges to $x_\infty$, uniformly in $\phi$, as $r\to\infty$, the map $f(re^{i\phi})$ converges to 0, uniformly in $\phi$, as $r\to \infty$. Hence by Proposition \ref{prop:lam d Morrey}(\ref{prop:lam W cpt}), the map $W^{1,p}_{\lam}(\C,\g^\C)\ni\al\mapsto \chi_{B_R^C}f\al\in L^p_{\lam}(\C,\g^\C)$ is compact. Definition \ref{defi:triv}(\ref{defi:triv split}) implies that $\chi_{B_R^C}S^2_{\phantom{2}1}=\chi_{B_R^C}f\pr$. It follows this map is compact. 

Finally, Proposition \ref{prop:lam d Morrey}(\ref{prop:lam Morrey}) and parts (\ref{defi:triv na}) and (\ref{defi:triv C}) of Definition \ref{defi:triv} imply that the map $\chi_{B_R^C}S^2_{\phantom{2}2}$ is compact. Claim \ref{claim:S i j compact} follows. This completes the proofs of statement {\bf (\ref{prop:X X w Fredholm})} and Proposition \ref{prop:X X w}.\end{pf}\end{proof}
\subsection{Proof of Theorem \ref{thm:L w * R} (Right inverse for $L_w^*$)}\label{subsec:proof:thm:L w * R}
Let $n\in\N$, $k,\ell\in\N\cup\{0\}$, $p>n$, $G$ a compact Lie group, $\lan\cdot,\cdot\ran_\g$ an invariant inner product on $\g:=\Lie(G)$, $\Om\sub \R^n$ an open subset, $P\to \Om$ a principal $G$-bundle and $A\in\A(P)$. Then $A$ and the standard metric on $\Om$ induce a connection $\na^A$ on $\wk(\g_P)$. For $\al\in W^{\ell,p}_\loc\big(\wk(\g_P)\big)$ we define
\[\Vert\al\Vert_{\ell,p,A}:=\Vert\al\Vert_{W^{\ell,p}_{A,\lan\cdot,\cdot\ran_\g}(\Om)}:=\sum_{i=0,\ldots,\ell}\Vert(\na^A)^\ell\al\Vert_{L^p(\Om)},\]
where the pointwise norms are taken with respect to $\lan\cdot,\cdot\ran_\g$. We denote 
\[W^{\ell,p}_A\big(\wk(\g_P)\big):=\big\{\al\in W^{\ell,p}_\loc\big(\wk(\g_P)\big)\,\big|\,\Vert\al\Vert_{\ell,p,A}<\infty\big\}.\]
For a subset $X\sub\R^n$ we denote by $\Int X$ its interior. For the proof of Theorem \ref{thm:L w * R} we need the following. 
\begin{prop}\label{prop:right} Let $n,p,G$ and $\lan\cdot,\cdot\ran_\g$ be as above and $K\sub \R^n$ a compact subset diffeomorphic to $\bar B_1$. Then for every $G$-bundle $P\to \Int K$ and $A\in\A(P)$ there exists a bounded right inverse $R$ of the operator 
\begin{equation}\label{eq:d A W}d_A^*:W^{1,p}_A\big(\wone(\g_P)\big)\to L^p(\g_P).\end{equation} 
Furthermore, there exist $\eps>0$ and $C>0$ such that for every principal $G$-bundle $P\to \Int K$ and every $A\in\A(P)$ satisfying $\Vert F_A\Vert_p\leq\eps$ the map $R$ can be chosen such that $\Vert R\Vert:=\sup\big\{\Vert R\xi\Vert_{1,p,A}\,\big|\,\xi\in L^p(\g_P):\,\Vert\xi\Vert_p\leq1\big\}\leq C$. 
\end{prop}
The proof of Proposition \ref{prop:right} is postponed to the appendix (page \pageref{sec:proof:prop:right}).
\begin{proof}[Proof of Theorem \ref{thm:L w * R}]\setcounter{claim}{0}\label{proof:thm:L w * R}\setcounter{claim}{0} Let $p,\lam$ and $w=(u,A)$ be as in the hypothesis. It follows from hypothesis (H) that there exists $\de>0$ such that $\mu^{-1}(\bar B_\de)\sub M^*$ (defined as in (\ref{eq:M *})). We define $c:=\inf\big\{|L_x\xi|/|\xi|\,\big|\,x\in\mu^{-1}(\bar B_\de),\,0\neq\xi\in\g\big\}$. It follows from Lemma \ref{le:si} that there exists a number $a>0$ such that $u(p)\in \mu^{-1}(\bar B_\de)$, for every $p\in\pi^{-1}\big(\C\wo (-a,a)^2\big)\sub P$. We choose constants $\eps_1$ and $C_1$ as in the second assertion of Proposition \ref{prop:right} (corresponding to $\eps$ and $C$, for $n=2$). Furthermore, we choose constants $\eps_2$ and $C_2$ as in Lemma \ref{le:infty 1 p A} (corresponding to $\eps$ and $C$). We define $\eps:=\min\{\eps_1,\eps_2\}$. By assumption we have $F_A\in L^p_\lam(\C)$. Hence there exists an integer $N>a$ such that $\Vert F_A\Vert_{L^p_\lam\big(\C\wo(-N,N)^2\big)}<\eps$. We choose a smooth map $\rho:[-1,1]\to[0,1]$ such that $\rho=0$ on $[-1,-3/4]\cup[3/4,1]$, $\rho=1$ on $[-1/4,1/4]$, and $\rho(-t)=\rho(t)$ and $\rho(t)+\rho(t-1)=1$, for all $t\in[0,1]$. We choose a bijection $(\phi,\psi):\Z\wo\{0\}\to \Z^2\wo\{-N,\ldots,N\}^2$. We define $\wt\rho:\R\to[0,1]$ by 
\[\wt\rho(t):=\left\{\begin{array}{ll}1,&\textrm{if } |t|\leq N,\\
\rho(|t|-N),&\textrm{if }N\leq|t|\leq N+1,\\
0,&\textrm{if }|t|\geq N+1.
\end{array}\right.\]
and $\rho_0:\R^2\to[0,1]$ by $\rho_0(s,t):=\wt\rho(s)\wt\rho(t)$. Furthermore, for $i\in\Z\wo\{0\}$ we define $\rho_i:\R^2\to[0,1]$ by $\rho_i(s,t):=\rho(s-\phi(i))\rho(t-\psi(i))$. We choose a compact subset $K_0\sub [-N-1,N+1]^2$ diffeomorphic to $\bar B_1$, such that $[-N-3/4,N+3/4]^2\sub \Int K_0$, and we denote $\Om_0:=\Int K_0$. Furthermore, we choose a compact subset $K\sub [-1,1]^2$ diffeomorphic to $\bar B_1$, such that $[-3/4,3/4]\sub \Int K$. For $i\in\Z\wo\{0\}$ we define $\Om_i:=\Int K+(\phi(i),\psi(i))$. For $i\in\Z$ we define $T_i:=d_A^*:W^{1,p}_A\big(\wone((P|_{\Om_i}\x\g)/G)\big)\to L^p_\lam\big((P|_{\Om_i}\x\g)/G\big)$. By the first assertion of Proposition \ref{prop:right} there exists a bounded right inverse $R_0$ of $T_0$. We fix $i\in\Z\wo\{0\}$. Since $\lam>1-2/p>0$, by our choice of $N$ we have $\Vert F_A\Vert_{L^p(\Om_i)}\leq\Vert F_A\Vert_{L^p_\lam(\Om_i)}<\eps$. Hence it follows from the statement of Proposition \ref{prop:right} that there exists a right inverse $R_i$ of $T_i$, satisfying 
\begin{equation}\label{eq:R i 1 p}\Vert R_i\xi\Vert_{W^{1,p}_A(\Om_i)}\leq C_1\Vert\xi\Vert_{L^p(\Om_i)},\quad \forall \xi\in W^{1,p}(\g_P|_{\Om_i}). 
\end{equation}
We define 
\begin{equation}\label{eq:wt R}\hat R:L^p_\loc(\g_P)\to W^{1,p}_\loc\big(\wone(\g_P)\big),\quad \hat R\xi:=\sum_{i\in\Z}\rho_i\cdot R_i(\xi|_{\Om_i}).\end{equation}
Each section $\xi:\R^2\to\g_P$ induces a section $L_u\xi:\R^2\to TM^u$. For $p\in\pi^{-1}\big(\R^2\wo(-N,N)^2\big)\sub P$ we have $u(p)\in\mu^{-1}(\bar B_\de)\sub M^*$, and therefore the map $L_{u(p)}^*L_{u(p)}:\g\to\g$ is invertible. Hence we may define 
\[\wt R:L^p_\loc(\g_P)\to L^p_\loc(TM^u),\quad (\wt R\xi)(z):=L_u(L_u^*L_u)^{-1}\big(\xi-d_A^*\hat R\xi\big)(z),\]
for $z\in\R^2\wo(-N,N)^2$, and $(\wt R\xi)(z):=0$, for $z\in (-N,N)^2$. Furthermore, we define $R:L^p_\loc(\g_P)\to L^p_\loc\big(TM^u\oplus\wone(\g_P)\big)$ by $R\xi:=(\wt R\xi,-\hat R\xi)$. It follows that $L_w^*R=\id$. Theorem \ref{thm:L w * R} is now a consequence of the following:
\begin{claim}\label{claim:ze X w} $R$ restricts to a bounded map from $L^p_\lam(\g_P)$ to $\XX_w^{p,\lam}$. 
\end{claim}
\begin{pf}[Proof of Claim \ref{claim:ze X w}] We choose a constant $C_3$ so big that $\sup_{z\in \Om_i}\lan z\ran^{p\lam}\leq C_3\inf_{z\in \Om_i}\lan z\ran^{p\lam}$, for every $i\in\Z$. For a weakly differentiable section $\xi:\R^2\to \g_P$ we denote $\Vert\xi\Vert_{1,p,\lam,A}:=\Vert\xi\Vert_{p,\lam}+\Vert d_A\xi\Vert_{p,\lam}$.

\begin{claim}\label{claim:d A * wt al} There exists a constant $C_4$ such that $\Vert \xi-d_A^*\hat R\xi\Vert_{1,p,\lam,A}\leq C_4\Vert\xi\Vert_{p,\lam}$, for every $\xi\in L^p_\lam(\g_P)$. 
\end{claim}
\begin{proof}[Proof of Claim \ref{claim:d A * wt al}] Let $\xi\in L^p_\lam(\g_P)$. We denote $\al_i:=R_i(\xi|_{\Om_i})$ and $\al:=\hat R\xi$. Since $\sum_{i\in\Z}\rho_i=1$, a straight-forward calculation shows that 
\begin{equation}\label{eq:d A * wt al}d_A^*\al=\xi-\sum_{i\in\Z}*\big((d\rho_i)\wedge*\al_i\big).\end{equation} 
Fix $z\in\R^2$. Then $\big|\big\{i\in\Z\,|\,\rho_i(z)\neq0\big\}\big|\leq4$. Hence (\ref{eq:d A * wt al}) implies that
\begin{equation}\label{eq:d A * p}\big|\big(\xi-d_A^*\al\big)(z)\big|^p\leq 4^{p-1}\Vert\rho'\Vert_\infty^p\sum_{i\in\Z}|\al_i(z)|^p.
\end{equation}
Inequalities (\ref{eq:d A * p}) and (\ref{eq:R i 1 p}) imply that 
\begin{equation}\nn\label{eq:xi d A * al}\Vert\xi-d_A^*\al\Vert_{p,\lam}^p\leq 4^p\Vert\rho'\Vert_\infty^p\max\big\{C_1^p,\Vert R_0\Vert^p\big\}C_3\sum_{i\in\Z}\Vert\xi\Vert_{L^p_\lam(\Om_i)}^p.\end{equation}
Equality (\ref{eq:d A * wt al}) implies that 
\begin{equation}\nn\label{eq:d A d A *}\big|d_A\big(\xi-d_A^*\al\big)(z)\big|^p\leq 8^{p-1}\max\big\{\Vert\rho''\Vert_\infty^p,\Vert\rho'\Vert_\infty^p\big\}\sum_{i\in\Z}\big(|\al_i|^p+|\na^A\al_i|^p\big)(z).\end{equation}
Combining this with (\ref{eq:R i 1 p}), Claim \ref{claim:d A * wt al} follows.
\end{proof}
We choose $C_4$ as in Claim \ref{claim:d A * wt al}. Let $\xi\in L^p_\lam(\g_P)$. We abbreviate $\wt\xi:=\xi-d_A^*\hat R\xi$. By the fact $\wt\xi|_{(-N,N)^2}=0$, Lemma \ref{le:infty 1 p A}, the fact $\lam>1-2/p>0$ and Claim \ref{claim:d A * wt al}, we have
\begin{equation}\label{eq:wt v infty}\Vert\wt\xi\Vert_\infty\leq C_2\Vert\wt\xi\Vert_{1,p,\lam,A}\leq C_2C_4\Vert\xi\Vert_{p,\lam}.\end{equation} 
Recall that $\Pr:TM\to TM$ and ${\Pr}^u:TM^u\to TM^u$ denote the orthogonal projections onto $\im L$ and $(u^*\im L)/G$. Claim \ref{claim:ze X w} is now a consequence of the following three claims. 
\begin{claim}\label{claim:hat R wt R} We have
\[\sup\big\{\big\Vert R\xi\big\Vert_\infty+\big\Vert |d\mu(u)\wt R\xi|+|{\Pr}^u \wt R\xi|+|\hat R\xi|\big\Vert_{p,\lam}\,\big|\,\xi\in L^p_\lam(\g_P):\,\Vert\xi\Vert_{p,\lam}\leq1\big\}<\infty.\]
\end{claim}
\begin{proof}[Proof of Claim \ref{claim:hat R wt R}] Let $\xi\in L^p_\lam(\g_P)$ be such that $\Vert\xi\Vert_{p,\lam}\leq1$. For $i\in\Z$ we denote $\al_i:=R_i(\xi|_{\Om_i})$. Since $\big|\big\{i\in\Z\,\big|\,\rho_i(z)\neq0\big\}\big|\leq4$, for every $z\in\R^2$, we have $\Vert\hat R\xi\Vert_{p,\lam}^p\leq4^p\sum_i\Vert \al_i\Vert_{L^p_\lam(\Om_i)}^p$. Using (\ref{eq:R i 1 p}), it follows that $\Vert\hat R\xi\Vert_{p,\lam}^p\leq4^p\max\{C_1^p,\Vert R_0\Vert^p\}C_3$. We define $C:=\max\big\{|d\mu(x)|+|\Pr_x|\,\big|\,x\in\BAR{u(P)}\big\}$. By Remark \ref{rmk:c} below, the statement of Claim \ref{claim:d A * wt al} and the fact $\wt R\xi|_{(-N,N)^2}=0$, we have
\[\big\Vert |d\mu(u)\wt R\xi|+|{\Pr}^u \wt R\xi|\big\Vert_{p,\lam}\leq Cc^{-1}\Vert\xi-d_A^*\hat R\xi\Vert_{L^p_\lam(\R^2\wo (-N,N)^2)}\leq Cc^{-1}C_4. \]
Inequality (\ref{eq:wt v infty}) and Remark \ref{rmk:c} imply that $\Vert \wt R\xi\Vert_\infty\leq c^{-1}\Vert\wt\xi\Vert_{L^\infty(\R^2\wo (-N,N)^2)}\leq c^{-1}C_2C_3$.  
We fix $i\in\Z\wo\{0\}$. Lemma \ref{le:infty 1 p A}, (\ref{eq:R i 1 p}) and the fact $\lam>1-2/p>0$ imply that $\Vert \al_i\Vert_\infty\leq C_2\Vert\al_i\Vert_{W^{1,p}_A(\Om_i)}\leq C_2C_1\Vert\xi\Vert_{L^p(\Om_i)}\leq C_2C_1$. Furthermore, 
 \[\Vert\al_0\Vert_\infty\leq C':=\sup\big\{\Vert\al\Vert_\infty\,\big|\,\al\in W^{1,p}_A\big(\wone((P|_{\Om_0}\x \g)/G\big):\,\Vert\al\Vert_{1,p,A}\leq1\big\}\Vert R_0\Vert.\]
It follows that $\Vert\hat R\xi\Vert_\infty\leq\sup_i\Vert \al_i\Vert_\infty\leq \max\{C_2C_1,C'\}<\infty$. Claim \ref{claim:hat R wt R} follows.  
\end{proof}
\begin{claim}\label{claim:na A p lam} We have $\sup\big\{\Vert\na^A(\wt R\xi)\Vert_{p,\lam}\,\big|\,\xi\in L^p_\lam(\g_P):\,\Vert\xi\Vert_{p,\lam}\leq1\big\}<\infty$. 
\end{claim}
\begin{proof}[Proof of Claim \ref{claim:na A p lam}] Let $\xi\in L^p_\lam(\g_P)$. We define $\wt\xi:=\xi-d_A^*\hat R\xi$, $\eta:=(L_u^*L_u)^{-1}\wt\xi$ and $\rho\in\Om^2(M,\g)$ as in (\ref{eq:xi rho v v'}). By Lemma \ref{le:na A L u} below, we have 
\begin{equation}\label{eq:na A L u eta}\na^A(L_u\eta)=L_ud_A\eta+\na_{d_Au}X_\eta.\end{equation} 
Using the second part of Lemma \ref{le:na A L u} (with $v:=L_u\eta$), it follows that 
\begin{equation}\label{eq:d A wt xi}L_u^*L_ud_A\eta =d_A\wt\xi-\rho(d_Au,L_u\eta)-L_u^*\na_{d_Au}X_\eta.\end{equation}
We choose a constant $C$ so big that $|\rho(v,v')|\leq C|v\Vert v'|$ and $|\na_vX_{\xi_0}|\leq C|v\Vert\xi_0|$, for every $x\in\mu^{-1}(\bar B_\de)$, $v,v'\in T_xM$ and $\xi_0\in\g$. We define $C:=\max\big\{c^{-1},3Cc^{-2}\big\}$. Since $\wt R\xi=L_u\eta$, equalities (\ref{eq:na A L u eta},\ref{eq:d A wt xi}) imply that 
\[\Vert\na^A(\wt R\xi)\Vert_{p,\lam}\leq C\big(\big\Vert d_A\wt\xi\big\Vert_{p,\lam}+\Vert d_Au\Vert_{p,\lam}\big\Vert\wt\xi\big\Vert_\infty\big).\]
Here we used Remark \ref{rmk:c}. Since $\Vert d_Au\Vert_{p,\lam}<\infty$, Claim \ref{claim:d A * wt al} and (\ref{eq:wt v infty}) now imply Claim \ref{claim:na A p lam}.
\end{proof}
\begin{claim}\label{claim:na A al p lam} We have $\sup\big\{\Vert\na^A(\hat R\xi)\Vert_{p,\lam}\,\big|\,\xi\in L^p_\lam(\g_P):\,\Vert\xi\Vert_{p,\lam}\leq1\big\}<\infty$.
\end{claim}
\begin{pf}[Proof of Claim \ref{claim:na A al p lam}] Let $\xi\in L^p_\lam(\g_P)$ be such that $\Vert\xi\Vert_{p,\lam}^p\leq1$. We write $\al_i:=R_i(\xi|_{\Om_i})$. Then $\na^A(\hat R\xi)=\sum_i\big(\rho_i\na^A\al_i + d\rho_i\otimes\al_i\big)$. Setting $C:=8^p\Vert\rho'\Vert_\infty^pC_3\max\{C_1^p,\Vert R_0\Vert^p\}$, it follows that
\begin{equation}\nn\label{eq:na A al p lam}\Vert\na^A(\hat R\xi)\Vert_{p,\lam}^p\leq8^{p-1}\sum_i\Big(\Vert\na^A\al_i\Vert_{p,\lam}^p + \Vert\rho'\Vert_\infty^p\Vert\al_i\Vert_{p,\lam}^p\Big)\leq C.\end{equation} 
Here in the second inequality we used the fact $\Vert\rho'\Vert_\infty\geq1$, and (\ref{eq:R i 1 p}). This proves Claim \ref{claim:na A al p lam}, and completes the proofs of Claim \ref{claim:ze X w} and Theorem \ref{thm:L w * R}. \end{pf}\end{pf}
\end{proof}
\appendix
\section{Weighted spaces and a Hardy-type inequality}\label{sec:weighted}
Let $d\in \Z$. The following lemma is used in section \ref{subsec:proof:Fredholm aug} in order to define a norm on $\XX_d$. If $d<0$ then let $\rho_0\in C^\infty(\R^2,[0,1])$ be such that $\rho_0(z)=0$ for $|z|\leq1/2$ and $\rho_0(z)=1$ for $|z|\geq1$. In the case $d\geq0$ we set $\rho_0:=1$. Recall that $p_d:\C\to \C$, $p_d(z):=z^d$. 
\begin{lemma}\label{le:X d iso} For every $1<p<\infty$ and $\lam>-2/p$ the map
\[\C\oplus L^{1,p}_{\lam-d}(\R^2,\C)\to \C\cdot\rho_0p_d+L^{1,p}_{\lam-d}(\R^2,\C),\quad (v_\infty,v)\mapsto v_\infty\rho_0p_d+v\]
is an isomorphism of vector spaces.  
\end{lemma}
\begin{proof}[Proof of Lemma \ref{le:X d iso}]\setcounter{claim}{0} This follows from a straight-forward argument. 
\end{proof}
The following proposition was used in the proof of Theorem \ref{thm:Fredholm aug} and Proposition \ref{prop:X X w}. For every normed vector space $V$ we denote by $C_b(\R^n,V)$ the space of bounded continuous maps from $\R^n$ to $V$. We denote $B_1^C:=\R^n\wo B_1$. 
\begin{prop}[Weighted Sobolev spaces]\label{prop:lam d Morrey} Let $n\in\N$. Then the following statements hold.
\begin{enui} \item\label{prop:lam Morrey} Let $n<p<\infty$. Then for every $\lam\in\R$ there exists $C>0$ such that 
\begin{equation}
    \label{eq:lam Morrey}\Vert u\lan\cdot\ran^{\lam+\frac np}\Vert_{L^\infty(\R^n)}\leq C\Vert u\Vert_{L^{1,p}_\lam(\R^n)},\quad \forall u\in W^{1,1}_\loc(\R^n).
\end{equation}
If $\lam>-n/p$ then $L^{1,p}_\lam(\R^n)$ is compactly contained in $C_b(\R^n)$. 
\item\label{prop:lam W} For every $k\in\N\cup\{0\}$, $1<p<\infty$ and $\lam\in\R$ the map $W^{k,p}_\lam(\R^n)\ni u\mapsto \lan\cdot\ran^\lam u\in W^{k,p}(\R^n)$ is a well-defined isomorphism (of normed spaces).
\item\label{prop:lam W cpt} Let $p>1$, $\lam\in\R$, and $f\in L^\infty(\R^n)$ be such that $\Vert f\Vert_{L^\infty(\R^n\wo B_i)}\to 0$, for $i\to \infty$. Then the map $W^{1,p}_\lam(\R^n)\ni u\mapsto fu\in L^p_\lam(\R^n)$ is compact.
\item \label{prop:lam d iso} For every $1<p<\infty$, $\lam\in\R$, $d\in\Z$ and $u\in L^{1,p}_\lam(B_1^C)$ we have
\begin{equation}\nn\label{eq:lam d }\Vert p_du\Vert_{L^{1,p}_{\lam-d}(B_1^C)}\leq \max\big\{-d2^{(-d+3)/2},2\big\}\Vert u\Vert_{L^{1,p}_\lam(B_1^C)}.
\end{equation}
\end{enui}
\end{prop}

\begin{proof} [Proof of Proposition \ref{prop:lam d Morrey}]\setcounter{claim}{0} {\bf Proof of (\ref{prop:lam Morrey}):} Inequality (\ref{eq:lam Morrey}) follows from inequality (1.11) in Theorem 1.2 in the paper by R.~Bartnik \cite{Ba}. Assume now that $\lam>-n/p$. Then it follows from Morrey's embedding theorem that there exists a canonical bounded inclusion $L^{1,p}_\lam(\R^n)\inj C_b(\R^n)$. In order to show that this inclusion is compact, let $u_\nu\in L^{1,p}_\lam(\R^n)$ be a sequence such that $C:=\sup_\nu\Vert u_\nu\Vert_{L^{1,p}_\lam(\R^n)}<\infty$. By Kondrachov's compactness theorem on $\bar B_j$ (for $j\in\N$), and a diagonal subsequence argument there exists a subsequence $u_{\nu_j}$ of $u_\nu$ that converges to some map $u\in W^{1,p}_\loc(\R^n)$, weakly in $W^{1,p}(B_j)$, and strongly in $C(\bar B_j)$, for every $j\in\N$. 
\begin{claim}\label{claim:u u nu j} We have $u\in C_b(\R^n)$ and $u_{\nu_j}$ converges to $u$ in $C_b(\R^n)$. 
\end{claim}
\begin{pf}[Proof of Claim \ref{claim:u u nu j}] We choose a constant $C'$ as in the first part of (\ref{prop:lam Morrey}). For every $R>0$ we have $\Vert u\Vert_{L^{1,p}_\lam(B_R)}\leq\limsup_j\Vert u_{\nu_j}\Vert_{L^{1,p}_\lam(B_R)}\leq C$. Hence $u\in L^{1,p}_\lam(\R^n)$.  Since $\lam>-n/p$, by inequality (\ref{eq:lam Morrey}), this implies $u\in C_b(\R^n)$. To see the second statement, we choose a smooth function $\rho:\R^n\to[0,1]$ such that $\rho(x)=0$ for $x\in B_1$, $\rho(x)=1$ for $x\in B_3^C$, and $|D\rho|\leq1$. Let $R\geq1$. We define $\rho_R:=\rho(\cdot/R):\R^n\to [0,1]$. Let $j\in\N$. Abbreviating $v_j:=u_{\nu_j}-u$, we have 
\begin{equation}\label{eq:Vert v j}\Vert v_j\Vert_\infty\leq\big\Vert v_j(1-\rho_R)\big\Vert_\infty+\big\Vert v_j\rho_R\big\Vert_\infty.\end{equation} 
Inequality (\ref{eq:lam Morrey}) implies that 
\begin{equation}\label{eq:Vert v j rho R}\big\Vert v_j\rho_R\big\Vert_\infty\leq C'R^{-\lam-\frac np}\Vert v_j\rho_R\Vert_{1,p,\lam}. 
\end{equation} 
Furthermore, $\Vert v_j\rho_R\Vert_{1,p,\lam}\leq 2\Vert v_j\Vert_{1,p,\lam}\leq4C$. Combining this with (\ref{eq:Vert v j}) and (\ref{eq:Vert v j rho R}), and the fact $\lim_{j\to\infty}\Vert v_j\Vert_{L^\infty(B_{3R})}=0$, it follows that $\limsup_{j\to\infty}\Vert v_j\Vert_\infty\leq 4CC'R^{-\lam-\frac np}$. Since $\lam>-n/p$ and $R\geq1$ is arbitrary, it follows that $u_{\nu_j}$ converges to $u$ in $C_b(\R^n)$. This proves Claim \ref{claim:u u nu j} and completes the proof of statement {\bf (\ref{prop:lam Morrey})}.
\end{pf}

{\bf Statement (\ref{prop:lam W})} follows from a straight-forward calculation.

{\bf Proof of (\ref{prop:lam W cpt}):} Let $f\in L^\infty(\R^n)$ be as in the hypothesis. Let $u_\nu\in W^{1,p}_\lam(\R^n)$ be a sequence such that $C:=\sup_\nu\Vert u_\nu\Vert_{W^{1,p}_\lam(\R^n)}<\infty$. By Kondrachov's theorem on $\bar B_j$ (for $j\in\N$) and a diagonal subsequence argument there exists a subsequence $(\nu_j)$ and a map $v\in L^p_\loc(\R^n)$, such that $fu_{\nu_j}$ converges to $v$, strongly in $L^p(K)$, as $j\to\infty$, for every compact subset $K\sub \R^n$. Standard arguments show that $v\in L^p_\lam(\R^n)$ and $fu_{\nu_j}$ converges to $v$ in $L^p_\lam(\R^n)$. This proves {\bf (\ref{prop:lam W cpt})}. 

{\bf Statement (\ref{prop:lam d iso})} follows from a straight-forward calculation. This completes the proof of Proposition \ref{prop:lam d Morrey}.\end{proof}
The next result was used in the proof of Proposition \ref{prop:X X w} and will be used in the proof of Lemma \ref{le:si}.
\begin{prop}[Hardy-type inequality]\label{prop:Hardy} Let $n\in\N$, $p>n$, $\lam>-n/p$ and $u\in W^{1,1}_\loc(\R^n,\R)$ be such that $\Vert Du|\cdot|^{\lam+1}\Vert_{L^p(\R^n)}<\infty$. Then $u(rx)$ converges to some $y_\infty\in \R$, uniformly in $x\in S^{n-1}$, as $r\to\infty$, and 
  \begin{equation}
    \label{eq:Hardy y}\Vert(u-y_\infty)|\cdot|^\lam\Vert_{L^p(\R^n)}\leq p/(\lam+n/p)\Vert Du|\cdot|^{\lam+1}\Vert_{L^p(\R^n)}.
  \end{equation}
\end{prop}
For the proof of Proposition \ref{prop:Hardy}, we need the following.
\begin{lemma}[Hardy's inequality]\label{le:Hardy} Let $n\in\N$, $1<p<\infty$, $\lam>-n/p$ and $u\in W^{1,1}_\loc(\R^n,\R)$. If there exists $R>0$ such that $u|_{B_R^C}=0$ then $\Vert u|\cdot|^\lam\Vert_{L^p(\R^n)}\leq p/(\lam+n/p)\Vert Du|\cdot|^{\lam+1}\Vert_{L^p(\R^n)}\,(\in[0,\infty])$. 
\end{lemma}
\begin{proof}[Proof of Lemma \ref{le:Hardy}]\setcounter{claim}{0} If $u$ is smooth then the stated inequality follows from Exercise 21, Chapter 6, in the book by O.~Kavian \cite{Ka}. The general case can be reduced to this case by mollifying the function $u$. This proves the lemma. \end{proof}
\begin{proof}[Proof of Proposition \ref{prop:Hardy}] \setcounter{claim}{0} Let $n,p,\lam$ as in the hypothesis. We set $\eps:=\lam+\frac np$. 
  \begin{claim}\label{claim:u x u y} There exists a constant $C_1$ such that for every weakly differentiable map $u:\R^n\to\R$ and $x,y\in\R^n$ satisfying $0<|x|\leq|y|$, we have $|u(x)-u(y)|\leq C_1|x|^{-\eps}\big\Vert D u|\cdot|^{\lam+1}\big\Vert_{L^p(B_{|x|}^C)}$. 
\end{claim}
\begin{proof}[Proof of Claim \ref{claim:u x u y}] By Morrey's theorem there is a constant $C$ such that $|u(0)-u(x)|\leq Cr^{1-n/p}\Vert Du\Vert_{L^p(B_r)}$, for every $r>0$, weakly differentiable $u:B_r\to\R$ and $x\in B_r$. Let $u,x$ and $y$ be as in the hypothesis of the claim. Let $N\in \N$ be such that $2^{N-1}|x|\leq |y|\leq 2^N|x|$. For $i=0,\ldots,N$ we define $x_i:=2^ix\in \R^n$. Furthermore, we set $x_{N+7}:=y|x_{N+7}|/|y|$ and $x_{N+8}:=y$, and we choose points $x_i\in S^{n-1}_{2^N|x|}$, $i=N+1,\ldots,N+6$, such that $|x_i-x_{i-1}|\leq 2^{N-1}|x|$, for $i=N+1,\ldots,N+7$. For $i=0,\ldots,N-1$ we have $x_i\in \bar B_{2^i|x|}(x_{i+1})$. Hence it follows from the statement of Morrey's theorem that $|u(x_{i+1})-u(x_i)|\leq C(2^i|x|)^{-\eps}\big\Vert Du|\cdot|^{\lam+1}\big\Vert_{L^p(B_{|x|}^C)}$. Moreover, for $i=N,\ldots,N+7$ we have $x_{i+1}\in \bar B_{2^{N-1}|x|}(x_i)$, and hence analogously, $|u(x_{i+1})-u(x_i)|\leq C(2^{N-1}|x|)^{-\eps}\Vert Du|\cdot|^{\lam+1}\Vert_{L^p(B_{|x|}^C)}$. Using the inequality $|u(y)-u(x)|\leq\sum_{i=0,\ldots,N+7}|u(x_{i+1})-u(x_i)|$, Claim \ref{claim:u x u y} follows. 
\end{proof}
Let $u\in W^{1,1}_\loc(\R^n,\R)$ be such that $\Vert Du|\cdot|^{\lam+1}\Vert_{L^p(\R^n)}<\infty$. It follows from Claim \ref{claim:u x u y} that there exists $y_\infty\in\R$ such that $u(rx)$ converges to $y_\infty$, as $r\to\infty$, uniformly in $x\in S^{n-1}$. To prove inequality (\ref{eq:Hardy y}), we choose a smooth map $\rho:[0,\infty)\to [0,1]$ such that $\rho(t)=1$ for $0\leq t\leq 1$, $\rho(t)=0$ for $t\geq2$ and $|\rho'(t)|\leq 2$. We fix a number $R>0$ and define $\rho_R:\R\to [0,1]$ by $\rho_R(x):=\rho(|x|/R)$. We abbreviate $v:=u-y_\infty$. Using Lemma \ref{le:Hardy} with $u$ replaced by $\rho_Rv$, we have
\[\Vert v|\cdot|^\lam\Vert_{L^p(B_R)}\leq \big\Vert\rho_R v|\cdot|^\lam\big\Vert_{L^p(\R^n)}\leq p/(\lam+n/p)\big\Vert D(\rho_Rv)|\cdot|^{\lam+1}\big\Vert_{L^p(\R^n)}.\] 
Combining this with a calculation using Leibnitz' rule, it follows that 
\begin{equation}\label{eq:Vert v lam}\Vert v|\cdot|^\lam\Vert_{L^p(B_R)}\leq p/(\lam+n/p)\big(4\Vert v|\cdot|^\lam\Vert_{L^p(B_{2R}\wo B_R)}+\big\Vert Du|\cdot|^{\lam+1}\big\Vert_{L^p(\R^n)}\big).\end{equation} 
Claim \ref{claim:u x u y} implies that $|v(x)|\leq C_1|x|^{-\eps}\Vert Du|\cdot|^{\lam+1}\Vert_{L^p(B_R^C)}$, for $x\in B_R^C$. Using the equalities $\int_{B_{2R}\wo B_R}|x|^{-n}dx=\log 2|S^{n-1}|$ and $\eps=\lam+n/p$, it follows that 
\[\Vert v|\cdot|^\lam\Vert_{L^p(B_{2R}\wo B_R)}^p\leq C_1^p\log 2|S^{n-1}|\big\Vert Du|\cdot|^{\lam+1}\big\Vert^p_{L^p(B_R^C)}.\]
Inequality (\ref{eq:Hardy y}) follows by inserting this into (\ref{eq:Vert v lam}) and sending $R$ to $\infty$. This proves Proposition \ref{prop:Hardy}. \end{proof}
The next result will be used to prove Corollary \ref{cor:d L L} below, which in turn is used in the proof of Theorem \ref{thm:Fredholm aug}. For every $d\in\Z$ we define $P_d$ and $\bar P_d$ to be the spaces of polynomials in $z\in\C$ and $\bar z$ of degree \emph{less than} $d$. (Note that if $d\leq 0$ then $P_d=\{0\}$.) We abbreviate $L^{1,p}_\lam:=L^{1,p}_\lam(\C,\C)$, $L^p_\lam:=L^p_\lam(\C,\C)$, $\d_{\bar z}:=\d^\C_{\bar z}$ and $\d_z:=\d^\C_z$. Let $X$ be a normed vector space and $Y\sub X$ be a closed subspace. We denote by $X^*$ the dual space of $X$ and equip $X/Y$ with the quotient norm.
\begin{prop}[Fredholm property for $\d_{\bar z}$]\label{prop:d L L} For every $d\in\Z$, $1<p<\infty$ and $-2/p+1<\lam<-2/p+2$ the following conditions hold.
\begin{enui}\item\label{prop:d L L:Fredholm} The operator $T:=\d_{\bar z}:L^{1,p}_{\lam-1-d}\to L^p_{\lam-d}$ is Fredholm.
\item\label{prop:d L L:kernel} We have $\ker T=P_d$.
\item\label{prop:d L L:coker} The map $\bar P_{-d}\to \big(L^p_{\lam-d}/\im T\big)^*$, $u\mapsto \big(v+\im T\mapsto \int_\C uv\,ds\,dt\big)$ is well-defined and an isometric isomorphism. Here we equip $L^p_{\lam-d}/\im T$ with the quotient norm.
\end{enui}
\end{prop}
\begin{rmk}\label{rmk:X Y} For every normed vector space $X$ and every closed subspace $Y\sub X$ the map $\big\{\phi\in X^*\,|\,\phi(x)=0,\,\forall x\in Y\big\}\to (X/Y)^*$, $\phi\mapsto \big(x+Y\mapsto \phi(x)\big)$, is well-defined and an isometric isomorphism. This follows from a straight-forward argument. 
\end{rmk}
We denote by $\S$ the space of Schwartz functions on $\C$ and by $\S'$ the space of temperate distributions. By $\Hat:\S'\to\S'$ we denote the Fourier transform, and by $\Unhat:\S'\to\S'$ the inverse transformation. 
\begin{proof}[Proof of Proposition \ref{prop:d L L}] \label{proof:d L L}\setcounter{claim}{0} Let $d,p,\lam$ and $T$ be as in the hypothesis. {\bf We start by proving (\ref{prop:d L L:kernel}).} A calculation in polar coordinates shows that for every polynomial $u$ in $z$ we have
\begin{eqnarray}\label{eq:u L}u\in L^{1,p}_{{\lam-1}-d}\iff \deg u<d-\lam+1-2/p.
\end{eqnarray}
Hence our assumption $\lam<-2/p+2$ implies that $\ker T\cont P_d$.

{\bf We prove that $\ker T\sub P_d$.} Let $u\in \ker T$. Then $0=\hhat{\d_{\bar z}u}(\ze)=\frac i2\ze\hhat u$ (as temperate distributions). It follows that the support of $\hhat u$ is either empty or consists of the point $0\in\C$. Hence the Paley-Wiener theorem implies that $u$ is real analytic in the variables $s$ and $t$, where $z=s+it$, and there exists $N\in\N$ such that $\sup_{z\in\C}|u(z)|\lan z\ran^N<\infty$. (See for example Theorem IX.12 in Vol.~I of the book \cite{ReSi}.) Therefore, by Liouville's Theorem $u$ is a polynomial in the variable $z$. Since by our assumption $\lam>-2/p+1$, it follows from (\ref{eq:u L}) that $u\in P_d$. This proves that $\ker T\sub P_d$ and completes the proof of (\ref{prop:d L L:kernel}).

{\bf We prove (\ref{prop:d L L:Fredholm}) and (\ref{prop:d L L:coker}).} We define $p':=p/(p-1)$. We identify the spaces $L^{p'}_{-\lam+d}$ and $(L^p_{\lam-d})^*$ via the isometric isomorphism $u\mapsto \big(v\mapsto \int_\C uv\big)$. Then the adjoint operator $T^*$ is given by $T^*=\d_z:L^{p'}_{-\lam+d}\iso(L^p_{\lam-d})^*\to (L^{1,p}_{{\lam-1}-d})^*$, where the derivatives are taken in the sense of distributions. 
\begin{claim}\label{claim:T P} $\ker T^*=\bar P_{-d}$.
\end{claim}
\begin{proof}[Proof of Claim \ref{claim:T P}] For every polynomial $u$ in $\bar z$ we have
\begin{equation}
  \label{eq:u L *}u\in L^{p'}_{-\lam+d} \iff \deg u<-d+\lam-2/p'=-d+\lam-2+2/p.
\end{equation}
Our assumption $\lam>-2/p+1$ and (\ref{eq:u L *}) imply that $\ker T^*\cont\bar P_{-d}$. Furthermore, the inclusion $\ker T^*\sub \bar P_{-d}$ is proved analogously to the inclusion $\ker T\sub P_d$, using $\lam<-2/p+2$ and (\ref{eq:u L *}). This proves Claim \ref{claim:T P}.
\end{proof}
We apply now Theorem 4.3 in the paper by R.~B.~Lockhart \cite{Lockhart Fred} with $T$ (case $d\leq0$) or $T^*$ (case $d>0$). The hypotheses of that theorem are satisfied, since by our assumption $-2/p+1<\lam<-2/p+2$, and since the operator $T=\d_{\bar z}$ ($T^*$) has constant coefficients and is elliptic, in the sense that its principal symbol $\si_T:\C\to \C,\quad \si_T(\ze)= \frac12(\ze_1+i\ze_2)$ ($\si_{T^*}(\ze)=\frac12(\ze_1-i\ze_2)$) does not vanish on $S^1\sub\C$. Hence that theorem implies that in the case $d\leq0$ the map $T$ is Fredholm, and in the case $d>0$ the operator $T^*$ is Fredholm. It follows that $\im T$ is closed if $d\leq0$. On the other hand, if $d>0$ then $\im T^*$ is closed, hence the same holds for $\im T$. Statements (\ref{prop:d L L:coker}) and (\ref{prop:d L L:Fredholm}) follow now from statement (\ref{prop:d L L:kernel}), Remark \ref{rmk:X Y} applied with $X:=L^p_{\lam-d}$ and $Y:=\im T$, and Claim \ref{claim:T P}. This proves Proposition \ref{prop:d L L}.
\end{proof}
Let $d\in\Z$, $1<p<\infty$, $-2/p+1<\lam<-2/p+2$, and $\rho_0:\R^2\to [0,1]$ be a smooth function that vanishes on $B_{1/2}$ and equals 1 on $B_1^C$. We equip $\C\rho_0p_d+L^{1,p}_{{\lam-1}-d}$ with the norm induced by the isomorphism of Lemma \ref{le:X d iso}. This norm is complete. (See e.g. \cite{Lockhart PhD}.) 
\begin{cor}\label{cor:d L L} The following map is Fredholm, with real index $2+2d$: 
  \begin{equation}
    \label{eq:d bar z rho}\d_{\bar z}:\C\rho_0p_d+L^{1,p}_{{\lam-1}-d}\to L^p_{\lam-d}\,.
  \end{equation}
\end{cor}
\begin{proof}[Proof of Corollary \ref{cor:d L L}]\setcounter{claim}{0} The composition of the isomorphism of Lemma \ref{le:X d iso} with (\ref{eq:d bar z rho}) is given by $T+S:\C\oplus L^{1,p}_{{\lam-1}-d}\to L^p_{\lam-d}$, where $T(x_\infty,u):=\d_{\bar z}u$ and $S(x_\infty,u):=x_\infty(\d_{\bar z}\rho_0)p_d$. The map $T$ is the composition of the canonical projection $\pr:\C\oplus L^{1,p}_{{\lam-1}-d}\to L^{1,p}_{{\lam-1}-d}$ with the operator $\d_{\bar z}:L^{1,p}_{{\lam-1}-d}\to L^p_{\lam-d}$. Using Proposition \ref{prop:d L L}, it follows that $T$ is Fredholm of real index $2+2d$. Furthermore, $S$ is compact, since it equals the composition of the canonical projection $\C\oplus L^{1,p}_{{\lam-1}-d}\to\C$ (which is compact) with a bounded operator. Corollary \ref{cor:d L L} follows.
\end{proof}
The next result is used in the proof of Theorem \ref{thm:Fredholm aug}. Let $(V,\lan\cdot,\cdot\ran)$ be a finite dimensional hermitian vector space, $A,B:V\to V$ positive  linear maps, $\lam\in\R$ and $1<p<\infty$. We define
\[T_\lam:=\left(
  \begin{array}{cc}
\d_{\bar z}& A\\
B & \d_z
  \end{array}
\right):W^{1,p}_\lam(\C,V\oplus V)\to L^p_\lam(\C,V\oplus V).\]
\begin{prop}\label{prop:d A B d} The operator $T_\lam$ is Fredholm of index 0.
\end{prop}
For the proof of Proposition \ref{prop:d A B d} we need the following result.
\begin{prop}\label{prop:Calderon} Let $(V,\lan\cdot,\cdot\ran),p$ and $A$ be as above, and $n\in\N$. Then the map $-\La+A:W^{2,p}(\R^n,V)\to L^p(\R^n,V)$ is an isomorphism (of Banach spaces).
\end{prop}
\begin{proof}[Proof of Proposition \ref{prop:Calderon}]\setcounter{claim}{0} Consider first the case $\dim_\C V=1$ and $A=1$. We define $G:=(2\pi)^{\frac n2}\big(\langle\cdot\rangle^{-2}\big)\Unhat\in \S'$. The map $\S\ni u\mapsto G*u\in\S$ is well-defined. By Calder\'on's Theorem this map extends uniquely to an isomorphism
\begin{equation}
  \label{eq:L G u}L^p(\R^n,\C)\ni u\mapsto G*u\in W^{2,p}(\R^n,\C).
\end{equation}
(See Theorem 1.2.3. in the book \cite{Ad}.) Note that $(-\La+1)(G*u)=\big(\lan\cdot\ran^2(G*u)\Hat\big)\Unhat=u$, for every $u\in\S$. It follows that the inverse of (\ref{eq:L G u}) is given by $-\La+1:W^{2,p}(\R^n,\C)\to L^p(\R^n,\C)$. Hence this is an isomorphism.

The general case can be reduced to the above case by diagonalizing the map $A$. This proves Proposition \ref{prop:Calderon}.
\end{proof}
\begin{proof}[Proof of Proposition \ref{prop:d A B d}]\label{proof:d A B d} \setcounter{claim}{0} We abbreviate $L^p:=L^p(\C,V\oplus V)$, etc. 

{\bf Assume first that $\lam=0$.} We denote by $A^{1/2},B^{1/2}:V\to V$ the unique positive  linear maps satisfying $(A^{\frac12})^2=A$, $(B^{\frac12})^2=B$. We define 
\begin{eqnarray}\nn L:=\left(
    \begin{array}{cc}\d_{\bar z}&A^{\frac12}B^{\frac12}\\
B^{\frac12}A^{\frac12}&\d_z
    \end{array}
\right):W^{1,p}\to L^p.
\end{eqnarray}
A short calculation shows that 
\begin{equation}
  \label{eq:A B T}T_0=\big(A^{\frac12}\oplus B^{\frac12}\big)L\big(A^{-\frac12}\oplus B^{-\frac12}\big).
\end{equation}
\begin{claim}\label{claim:L} The operator $L$ is an isomorphism. 
\end{claim}
\begin{proof}[Proof of Claim \ref{claim:L}] We define
\[L':=\left(
  \begin{array}{cc}-\d_z&A^{\frac12}B^{\frac12}\\
B^{\frac12}A^{\frac12}&-\d_{\bar z}\end{array}
\right):W^{2,p}\to W^{1,p}.\]
By a short calculation we have $LL'=\big(-\La/4 +A^{\frac12}B A^{\frac12}\big)\oplus\big(-\La/4 + B^{\frac12}A B^{\frac12}\big):W^{2,p}\to L^p$. Since the linear maps $A^{\frac12}BA^{\frac12},B^{\frac12}AB^{\frac12}:V\to V$ are positive, Proposition \ref{prop:Calderon} implies that $LL'$ is an isomorphism. We denote by $(LL')^{-1}:L^p\to W^{2,p}$ its inverse and define $R:=L'(LL')^{-1}:L^p\to W^{1,p}$. Then $R$ is bounded and $LR=\id_{L^p}$. 

By a short calculation, we have $LL'(u,v)=L'L(u,v)$, for every Schwartz function $(u,v)\in\S$. This implies that $(LL')^{-1}L|_\S=L(LL')^{-1}|_\S$, and therefore $RL|_\S=\id_\S$. Since $RL:W^{1,p}\to W^{1,p}$ is continuous and $\S\sub W^{1,p}$ is dense, it follows that $RL=\id_{W^{1,p}}$. Claim \ref{claim:L} follows. 
\end{proof}
The maps $A^{\frac12}\oplus B^{\frac12}:L^p\to L^p$ and $A^{-\frac12}\oplus B^{-\frac12}:W^{1,p}\to W^{1,p}$ are automorphisms. Therefore, (\ref{eq:A B T}) and Claim \ref{claim:L} imply that $T$ is an isomorphism. 

Consider now the {\bf general case $\lam\in\R$}. The map $L^p\ni (u,v)\mapsto \lan\cdot\ran^{-\lam}(u,v)\in L^p_\lam$ is an isometric isomorphism. Furthermore, by Proposition \ref{prop:lam d Morrey}(\ref{prop:lam W}) the map $W^{1,p}_\lam\ni (u,v)\mapsto \lan\cdot \ran^\lam(u,v)\in W^{1,p}$ is well-defined and an isomorphism. We define $S:=\lan\cdot\ran^\lam(\d_{\bar z}\lan\cdot\ran^{-\lam})\oplus\lan\cdot\ran^\lam(\d_z\lan\cdot\ran^{-\lam}):W^{1,p}\to L^p$. Direct calculations show that $T_\lam=\lan\cdot\ran^{-\lam}(T_0+S)\lan\cdot\ran^\lam$, $|\d_{\bar z}\lan\cdot\ran^{-\lam}|\leq |\lam|\lan\cdot\ran^{-\lam-1}/2$ and $|\d_z\lan\cdot\ran^{-\lam}|\leq|\lam|\lan\cdot\ran^{-\lam-1}/2$. Therefore, Proposition \ref{prop:lam d Morrey}(\ref{prop:lam W cpt}) implies that the operator $S$ is compact. We proved that $T_0$ is an isomorphism. It follows that $T_\lam$ is a Fredholm map of index 0. This proves Proposition \ref{prop:d A B d} in the general case.
\end{proof}
\section{Proof of Proposition \ref{prop:right} (Right inverse for $d_A^*$)}\label{sec:proof:prop:right}
For the proof of Proposition \ref{prop:right} we need the following three results. Let $n\in\N$, $1\leq p\leq\infty$, $G$ a compact Lie group with Lie algebra $\g$ and $\lan\cdot,\cdot\ran_\g$ an invariant inner product on $\g$. For a Riemannian manifold $X$ and a $G$-bundle $P\to X$ we denote by $\A^{1,p}(P)$ the space of $W^{1,p}$-connections on $P$.
\begin{prop}\label{prop:Uhlenbeck} Let $K\sub \R^n$ be a compact subset diffeomorphic to $\bar B_1$. If $n/2<p<\infty$ then there exist constants $\eps>0$ and $C$ such that for every principal $G$-bundle $P\to \Int K$, $A_0\in\A(P)$, and $A\in\A^{1,p}(P)$ the following holds. If $A_0$ is flat and $\Vert F_A\Vert_p\leq\eps$ then there exists a gauge transformation $g\in \G^{2,p}(P)$ such that $\Vert g^*A-A_0\Vert_{1,p,A_0}\leq C\Vert F_A\Vert_p$. 
\end{prop}
\begin{proof}[Proof of Proposition \ref{prop:Uhlenbeck}]\setcounter{claim}{0} Let $n,p,G,\lan\cdot,\cdot\ran_\g$ and $K$ be as in the hypothesis. We denote by $\wt A_0$ the trivial connection on $\Int K\x G$. By Uhlenbeck's gauge theorem there exist constants $\eps>0$ and $C$ such that for every connection $\wt A\in \A^{1,p}(\Int K\x G)$ satisfying $\Vert F_{\wt A}\Vert_p\leq\eps$ there exists $\wt g\in\G^{2,p}(\Int K\x G)$ such that $\Vert \wt g^*{\wt A}-\wt A_0\Vert_{1,p,\wt A_0}\leq C\Vert F_{\wt A}\Vert_p$. (This follows for example from Theorem 6.3 in \cite{Weh}.) Let $P\to \Int K$ be a $G$-bundle, $A_0\in\A(P)$ be flat, and $A\in\A^{1,p}(P)$ be such that $\Vert F_A\Vert_p\leq\eps$. Since $A_0$ and $\wt A_0$ are flat, there exists a smooth isomorphism of $G$-bundles $\Psi:\Int K\x G\to P$ (with fixed base) such that $\Psi^*A_0=\wt A_0$. We choose $\wt g\in\G^{2,p}(\Int K\x G)$ as in the conclusion of Uhlenbeck's theorem with $\wt A:=\Psi^*A$. We define $g:=\wt g\circ\Psi^{-1}\in\G^{2,p}(P)$. A straight-forward calculation shows that $\Vert g^*A-A_0\Vert_{1,p,A_0}=\Vert \wt g^*\wt A-\wt A_0\Vert_{1,p,\wt A_0}$. The statement of Proposition \ref{prop:Uhlenbeck} follows from this. 
\end{proof}
\begin{prop}\label{prop:right d A * d A} Let $n\in\N$, $1<p<\infty$, $G$ a compact Lie group with Lie algebra $\g$, $\lan\cdot,\cdot\ran_\g$ an invariant inner product on $\g$, and $K\sub \R^n$ a compact subset diffeomorphic to $\bar B_1$. Then there exists a constant $C$ such that for every principal $G$-bundle $P\to \Int K$ and every smooth flat connection $A$ on $P$ there exists right inverse $R$ of $d_A^*d_A:W^{2,p}_A(\g_P)\to L^p(\g_P)$ satisfying $\Vert R\Vert:=\big\{\Vert R\xi\Vert_{2,p,A}\,\big|\,\xi\in L^p(\g_P):\,\Vert\xi\Vert_p\leq1\big\}\leq C$.
\end{prop}
\begin{proof}[Proof of Proposition \ref{prop:right d A * d A}]\setcounter{claim}{0} Let $n\in\N$ and $1<p<\infty$. For an open subset $U\sub\R^n$ we denote by $C^\infty_0(U)$ the compactly supported smooth functions on $U$. We define the map $\wt T:C^\infty_0(B_1)\to C^\infty(B_1)$ as follows. We denote by $\Phi:\R^n\wo\{0\}\to \R$ the fundamental solution of the Laplace equation (see e.g. \cite{Ev}, p.~22). Let $f\in C^\infty_0(B_1)$. We define $\wt f:\R^n\to\R$ to be the extension of $f$ by 0 outside $B_1$. We denote by $*$ convolution in $\R^n$ and define $\wt Tf:=(\Phi*\wt f)|_{B_1}$. Note that $\Phi$ is locally integrable, hence the convolution is well-defined. Furthermore, $\wt Tf$ is smooth, and $\La\wt Tf=f$. (The first assertion follows from differentiation under the integral, and for the second see for example Theorem 1 on p.~23 in the book \cite{Ev}.) 
\begin{claim}\label{claim:C Tf} There exists a constant $C$ such that $\Vert \wt Tf\Vert_{W^{2,p}(B_1)}\leq C\Vert f\Vert_{L^p(B_1)}$, for every $f\in C^\infty_0(B_1)$. 
\end{claim}
\begin{proof}[Proof of Claim \ref{claim:C Tf}] Young's inequality states that $\Vert \wt Tf\Vert_{L^p(B_1)}\leq\Vert \Phi\Vert_{L^1(B_2)}\Vert f\Vert_{L^p(B_1)}$, for every $f\in C^\infty_0(B_1)$. Furthermore, the Calder\'on-Zygmund inequality states that there exists a constant $C$ such that for every $f\in C^\infty_0(\R^n)$ we have $\Vert D^2(\Phi*f)\Vert_p\leq C\Vert f\Vert_p$. (See for example Theorem B.2.7 in \cite{MSJ}. Note that $\Phi_j*f=(\d_j\Phi)*f=\d_j(\Phi*f)$.) Claim \ref{claim:C Tf} follows from this. 
\end{proof}
We fix a constant $C$ as in Claim \ref{claim:C Tf}. By this claim the map $\wt T$ uniquely extends to a bounded linear map $T: L^p(B_1)\to W^{2,p}(B_1)$. Since $\La \wt Tf=f$, for every $f\in C^\infty_0(B_R)$, a density argument shows that $\La Tf=f$, for every $f\in W^{2,p}(B_1)$, \ie $T$ is a right inverse for $\La:W^{2,p}(B_R)\to L^p(B_R)$. 

Let now $G,\lan\cdot,\cdot\ran_\g$ and $K$ be as in the hypothesis. Without loss of generality we may assume that $K=\bar B_1$. Let $\pi:P\to B_1$ be a $G$-bundle, and $A\in\A(P)$ be flat. We fix a point $p_0\in P_0$, and denote by $\si:B_1\to P$ the $A$-horizontal section through $p_0$. This is the unique smooth section of $P$ satisfying $A\,d\si=0$ and $\si(0)=p_0$. For $k\geq0$ we define the map $\Psi_k:W^{k,p}(B_1,\g)\to W^{k,p}_{A}(B_1,\g_P)$ by $\Psi_k\xi:=G\cdot(\si,\xi)$. This is an isometric isomorphism. We define $R:=-\Psi_2T\Psi_0^{-1}$. It follows that $\big\Vert R\big\Vert\leq\Vert\Psi_2\Vert \Vert T\Vert \Vert\Psi_0^{-1}\Vert=\Vert T\Vert\leq C$, where we use various operator norms. A straight-forward calculation shows that $d_A^*d_A\Psi_2=\Psi_0d^*d=-\Psi_0\La:W^{2,p}(B_1,\g)\to L^p(B_1,\g_P)$. It follows that $R$ is a right inverse for $d_{A}^*d_{A}$. This proves Proposition \ref{prop:right d A * d A}.
\end{proof}
\begin{lemma}\label{le:infty 1 p A} Let $n\in\N$, $p>n$, $G$ a compact Lie group with Lie algebra $\g$, $\lan\cdot,\cdot\ran_\g$ an invariant inner product on $\g$ and $K\sub\R^n$ a compact subset diffeomorphic to $\bar B_1$. Then there exist constants $C$ and $\eps>0$ such that for every principal $G$-bundle $\pi:P\to \Int K$, $A\in\A^{1,p}(P)$, $k=0,1$ and $\al\in W^{1,p}_A\big(\wk(\g_P)\big)$ the following holds. If $\Vert F_A\Vert_{L^p(\Int K)}\leq\eps$ then $\Vert\al\Vert_{L^\infty(\Int K)}\leq C\Vert\al\Vert_{1,p,A}$. 
\end{lemma}
\begin{proof}[Proof of Lemma \ref{le:infty 1 p A}]\setcounter{claim}{0} Let $n,p,G,\g,\lan\cdot,\cdot\ran_\g$ and $K$ be as in the hypothesis. We choose constants $\eps>0$ and $C_1:=C$ as in Proposition \ref{prop:Uhlenbeck}. Since by assumption $p>n$, it follows from Morrey's theorem that there exists a constant $C_2$ with the following property. If $P\to\Int K$ is a $G$-bundle, $A_0\in\A(P)$ is flat, $k=0,1$ and $\al\in W^{1,p}_{A_0}\big(\wk(\g_P)\big)$ then $\Vert \al\Vert_\infty\leq C_2\Vert\al\Vert_{W^{1,p}_{A_0}}$. Let $\pi:P\to \Int K$ be a $G$-bundle and $A\in\A^{1,p}(P)$ be such that $\Vert F_A\Vert_{L^p(\Int K)}\leq\eps$. We choose $g\in\G^{2,p}(P)$ as in Proposition \ref{prop:Uhlenbeck}. Let $\al\in W^{1,p}_A\big(\wk(\g_P)\big)$. We set $A':=g^*A$, $\al':=g^*\al$ and $C_3:=\max\big\{[\xi,\eta]\,|\,\xi,\eta\in \g:\,|\xi|\leq1,\,|\eta|\leq1\big\}$. A direct calculation shows that $(\na^{A_0}-\na^{A'})\al'=[(A'-A_0)\otimes\al']$. It follows that 
\begin{equation}\label{eq:al infty}\Vert \al\Vert_\infty\leq\Vert \al'\Vert_\infty\leq C_2\Vert\al'\Vert_{1,p,A_0}\leq C_2\big(\Vert\al'\Vert_{1,p,A'}+C_3\Vert A'-A_0\Vert_\infty\Vert\al'\Vert_p\big).\end{equation}
By the statement of Proposition \ref{prop:Uhlenbeck}, we have $\Vert A'-A_0\Vert_{1,p,A_0}\leq C_1\eps$. Combining this with Morrey's theorem and (\ref{eq:al infty}), Lemma \ref{le:infty 1 p A} follows. 
\end{proof}
\begin{proof}[Proof of Proposition \ref{prop:right}]\setcounter{claim}{0} Let $n,p,G,\lan\cdot,\cdot\ran_\g,K,P$ and $A$ be as in the hypothesis. We show that the operator (\ref{eq:d A W}) admits a bounded right inverse.
\begin{claim}\label{claim:L} There exists a bounded linear map $L:W^{1,p}_A(\Int K,\g_P)\to W^{1,p}_A\big(\wone(\g_P)\big)$ such that $d_A^*L=\id$. 
\end{claim}
\begin{proof}[Proof of Claim \ref{claim:L}] We may assume without loss of generality that $K=\bar B_1$. We define $\Om\sub\R\x P$ to be the subset consisting of all $(t,p)$ such that $|t+x_1|^2+x_2^2+\cdots+x_n^2<1$, where $x:=\pi(p)\in B_1$. Furthermore, we denote by $\Psi:\Om\to P$ the $A$-parallel transport in $x_1$-direction. Let $(t_0,p_0)\in\Om$. We denote $x_0:=\pi(p_0)\in\bar B_1$. Then $\Psi(t_0,p_0)=p(t_0)$, where $p:\{t\in\R\,|\,(t,x_0)\in\Om\}\to P$ is the unique path satisfying $\pr\circ p(t)=x_0+(t,0,\ldots0)$, for every $t\in\R$, $A\dot p=0$ and $p(0)=p_0$. Let $\xi\in W^{1,p}(B_1,\g_P)$. We define $\wt\xi:P\to\g$ by the condition $[p,\wt\xi(p)]=\xi\circ \pi(p)$, for $p\in P$. Let $p\in B_1$ and $(x_1,\ldots,x_n):=\pi(p)$. We define $\wt\eta:=\int_{-x_1}^0\wt\xi\circ\Psi(t,p)dt\in\g$. Furthermore, we define the section $\eta:B_1\to\g_P$ by the condition $\eta\circ\pi(p)=[p,\wt\eta(p)]$, for every $p\in P$, and  $L\xi:=\eta\,dx^1$. Then $L$ has the required properties. 
This proves Claim \ref{claim:L}.
\end{proof}
We choose a map $L$ as in Claim \ref{claim:L}. Furthermore, we choose a smooth flat connection $A_0$ on $P$. By Proposition \ref{prop:right d A * d A} there exists a bounded right inverse $R_0$ of $d_{A_0}^*d_{A_0}:W^{2,p}_{A_0}(\g_P)\to L^p(\g_P)$. We define the map 
\begin{equation}\label{eq:R d A R 0}R:=d_AR_0+L\big(\id-d_A^*d_AR_0\big):L^p(\g_P)\to W^{1,p}_A\big(\wone(\g_P)\big).\end{equation} 
The first assertion of Proposition \ref{prop:right} is now a consequence of the following.
\begin{claim}\label{claim:R} The map (\ref{eq:R d A R 0}) is well-defined and bounded, and $d_A^*R=\id$. 
\end{claim}
\begin{proof}[Proof of Claim \ref{claim:R}] A short calculation shows that $S:=d_A^*d_A-d_{A_0}^*d_{A_0}$ is of first or zeroth order. Hence it defines a bounded linear map from $W^{2,p}_A(\Int K,\g_P)$ to $W^{1,p}_A(\Int K,\g_P)$. Furthermore $\id-d_A^*d_AR_0=-SR_0$. This implies that $R$ is well-defined and bounded. A short calculation shows that $d_A^*R=\id$. This proves Claim \ref{claim:R}. 
\end{proof}
To prove the second statement of Proposition \ref{prop:right}, we choose $\eps_1,C_1$ as in  Lemma \ref{le:infty 1 p A} (corresponding to $\eps,C$) and $\eps_2,C_2$ as in Proposition \ref{prop:Uhlenbeck} (corresponding to $\eps,C$). We also fix a constant $C_3:=C$ as in Proposition \ref{prop:right d A * d A}, and we define $\eps:=\min\big\{\eps_1,\eps_2,1/(2C_1C_2C_3)\big\}$. Let $P\to\Int K$ be a $G$-bundle and $A\in \A(P)$ be such that $\Vert F_A\Vert_p\leq\eps$. We choose a flat smooth connection $A_0$ on $P$. By the assertion of Proposition \ref{prop:Uhlenbeck} there exists $g\in\G^{2,p}(P)$ such that 
\begin{equation}\label{eq:g * A A 0}\Vert g^*A-A_0\Vert_{1,p,A_0}\leq C_2\Vert F_A\Vert_p.\end{equation}
Furthermore, by the assertion of Proposition \ref{prop:right d A * d A} there exists a right inverse $R_0$ of $d_{A_0}^*d_{A_0}:W^{2,p}_{A_0}(\g_P)\to L^p(\g_P)$ satisfying $\Vert R_0\Vert\leq C_3$, where $\Vert R_0\Vert$ denotes the operator norm of $R_0$. We define the map $S:W^{1,p}_{A_0}\big(\wone(\g_P)\big)\to L^p(\g_P)$ by $S\al:=-*\big[(g^*A-A_0)\wedge*\al\big]$. Since $A_0$ is flat, the hypotheses of Lemma \ref{le:infty 1 p A} are satisfied. Hence by this lemma and (\ref{eq:g * A A 0}), we have
\[\Vert S\al\Vert_p\leq\Vert g^*A-A_0\Vert_p\Vert\al\Vert_\infty\leq C_2C_1\eps\Vert\al \Vert_{1,p,A_0}(\g_P),\]
for every $\al\in W^{1,p}_{A_0}\big(\wone(\g_P)\big)$. Hence $S$ is well-defined, and
\[\Vert Sd_{A_0}R_0\Vert\leq\Vert S\Vert \Vert d_{A_0}\Vert \Vert R_0\Vert\leq C_2C_1\eps C_3\leq1/2,\] 
Hence $\id+Sd_{A_0}R_0:L^p(\g_P)\to L^p(\g_P)$ is invertible, and the von Neumann series $\sum_{i=0}^\infty (-Sd_{A_0}R_0)^i$ converges in the operator norm and equals $(\id+Sd_{A_0}R_0)^{-1}$. Furthermore, $\big\Vert\big(\id+Sd_{A_0}R_0\big)^{-1}\big\Vert\leq\sum_{i=0}^\infty2^{-i}=2$. We define $R:=g_*d_{A_0}R_0\big(\id+Sd_{A_0}R_0\big)^{-1}g^*:L^p(\g_P)\to W^{1,p}_{A_0}\big(\wone(\g_P)\big)$. Since $S=d_{g^*A}^*-d_{A_0}^*$, we have $d_A^*R=g_*d_{g^*A}^*g^*R=\id$. Furthermore, for every $\xi\in L^p(\g_P)$, we have
\[\Vert \na^A(R\xi)\Vert_p=\Vert \na^{g^*A}(g^*R\xi)\Vert_p\leq\Vert \na^{A_0}(g^*R\xi)\Vert_p+\big\Vert\big[(g^*A-A_0)\otimes g^*R\xi\big]\big\Vert_p\]
\[\leq (2+2C_4C_2\eps)C_3\Vert\xi\Vert_p\,,\]
where $C_4:=\max\big\{|[\al\otimes\be]|\,\big|\,\al,\be\in \wone(\g_P):\,|\al|,|\be|\leq1\big\}$. Here in the last step we used (\ref{eq:g * A A 0}) and the fact $\Vert F_A\Vert_p\leq\eps$. This proves the second statement and concludes the proof of Proposition \ref{prop:right}.
\end{proof}
\section{Other auxiliary results}\label{sec:aux}
The next lemma was used in the proof of Theorem \ref{thm:Fredholm}. Here for a linear map $D:X\to Y$ we denote $\coker D:=Y/\im D$. 
\begin{lemma}\label{le:X Y Z} Let $X,Y,Z$ be vector spaces and $D':X\to Y$ and $T:X\to Z$ be linear maps. We define $D:=D'|_{\ker T}$. Then the following holds.
  \begin{enui}\item\label{le:X Y Z ker} $\ker D=\ker(D', T)$. 
\item\label{le:X Y Z coker} The map $\Phi:\coker D\to \coker(D',T)$, $\Phi(y+\im D):=(y,0)+\im(D',T)$, is well-defined and injective. If $T:X\to Z$ is surjective then $\Phi$ is also surjective. 
\item\label{le:X Y Z closed} Let $\Vert\cdot\Vert_Y,$ $\Vert\cdot\Vert_Z$ be norms on $Y$ and $Z$ and assume that $\im(D',T)$ is closed in $Y\oplus Z$. Then $\im D$ is closed in $Y$. 
\end{enui}
\end{lemma}
The proof of Lemma \ref{le:X Y Z} is straight-forward and left to the reader. The following result was used in the proof of Proposition \ref{prop:triv}. We define the map $f:\C\wo\{0\}\to S^1$ by $f(z):=z/|z|$. For two topological spaces $X$ and $Y$ we denote by $C(X,Y)$ the set of all continuous maps from $X$ to $Y$, and by $[X,Y]$ the set of all (free) homotopy classes of such maps. Let $V$ be a finite dimensional complex vector space. We denote by $\End(V)$ the space of its (complex) endomorphisms of $V$, by $\det:\End(V)\to\C$ the determinant map, and by $\Aut(V)\sub \End(V)$ the group of automorphisms of $V$. 
\begin{lemma}\label{le:Aut} The map $C(S^1,\Aut(V))\to \Z$ given by $\Phi\mapsto \deg\big(f\circ\det\circ\Phi\big)$ descends to a bijection $\big[S^1,\Aut(V)\big]\to \Z$. 
\end{lemma}
\begin{proof}[Proof of Lemma \ref{le:Aut}]\setcounter{claim}{0} We choose a hermitian inner product $V$ and denote by $\U(V)$ the corresponding group of unitary automorphisms of $V$. The map $\det:\U(V)\to S^1$ induces an isomorphism of fundamental groups, see \eg Proposition 2.23 in the book by D.~McDuff and D.~A.~Salamon \cite{MSIntro}. Furthermore, the space $\Aut(V)$ strongly deformation retracts onto $\U(V)$. (This follows from the Gram-Schmidt orthonormalization procedure.) This implies that the map $f\circ\det:\Aut(V)\to S^1$ induces an isomorphism of fundamental groups. Let $\Phi_0\in\Aut(V)$. It follows that the map $C(S^1,\Aut(V))\to \Z$ given by $\Phi\mapsto \deg\big(f\circ\det\circ\Phi\big)$ descends to an isomorphism between $\pi_1(\Aut(V),\Phi_0)$ and $\Z$. Since this group is abelian, the map $\pi_1(\Aut(V),\Phi_0)\to \big[S^1,\Aut(V)\big]$ that forgets the base point $\Phi_0$, is a bijection. The statement of Lemma \ref{le:Aut} follows from this. 
\end{proof}
The next lemma was used in the proof of Theorem \ref{thm:L w * R}. Let $X$ and $M$ be manifolds, $G$ a Lie group with Lie algebra $\g$, $\lan\cdot,\cdot\ran_\g$ an invariant inner product on $\g$, $\lan\cdot,\cdot\ran_M$ a $G$-invariant Riemannian metric on $M$, and $\na$ its Levi-Civita connection. For $\xi\in\g$ we denote by $X_\xi$ the vector field on $M$ generated by $\xi$. We define the tensor $\rho:TM\oplus TM\to\g$ by 
\begin{equation}\label{eq:xi rho v v'}\lan\xi,\rho(v,v')\ran_\g:=\lan\na_vX_\xi,v'\ran_M.\end{equation}
A short calculation shows that $\rho$ is skew-symmetric. This two-form was introduced in \cite{Ga} (page 181). The next lemma corresponds to Proposition 7.1.3(a,b) in \cite{Ga}. Let $P\to X$ be a principal bundle, $A\in\A(P)$, $u\in C^\infty\big(X,(P\x M)/G\big)$, $v\in\Ga(TM^u)$, and $\xi\in \Ga(\g_P)$. We define the connection $\na^A$ on $TM^u\to X$ as in Section \ref{sec:back}. 
\begin{lemma}\label{le:na A L u} $\na^AL_u\xi-L_ud_A\xi=\na_{d_Au}X_\xi,\quad d_AL_u^*v-L_u^*\na^Av=\rho(d_Au,v).$ 
\end{lemma}
\begin{proof}[Proof of Lemma \ref{le:na A L u}]\setcounter{claim}{0} This follows from short calculations.
\end{proof}
Let $M,\om,G,\g,\lan\cdot,\cdot\ran_\g,\mu$ and $J$ be as in Section \ref{sec:main}, and $\lan\cdot,\cdot\ran_M:=\om(\cdot,J\cdot)$. The following remark was used in the proofs of Theorems \ref{thm:Fredholm} and \ref{thm:L w * R}. Recall the definition (\ref{eq:M *}) of $M^*\sub M$, and that $\Pr:TM\to TM$ denotes the orthogonal projection onto $\im L$. 
\begin{rmk}\label{rmk:c}Let $K\sub M^*$ be compact. We define $c:=\inf\big\{|L_x\xi|/|\xi|\,\big|\,x\in K,\,0\neq\xi\in\g\big\}$. Then $c>0$. Let $x\in K$. Then $L_x^*L_x$ is invertible, and
\begin{equation}\label{eq:L x * L x c}|(L_x^*L_x)^{-1}|\leq c^{-2},\,\big|L_x(L_x^*L_x)^{-1}\big|\leq c^{-1},\,L_x(L_x^*L_x)^{-1}L_x^*={\Pr}_x,\end{equation}
where the $|\cdot|$'s denote operator norms. Furthermore, $|\Pr_x v|\leq c^{-1}|L_x^*v|$, for every $v\in T_xM$. These assertions follow from short calculations.
\end{rmk}
Assume that hypothesis (H) holds. The following lemma was used in the proof of Proposition \ref{prop:X X w}. 
\begin{lemma}\label{le:U} There exists a neighborhood $U\sub M$ of $\mu^{-1}(0)$, such that
\begin{equation}\label{eq:c inf}c:=\inf\big\{\big|d\mu(x)L_x^\C\al\big|+|\Pr L_x^\C\al|\,\big|\,x\in U,\,\al\in\g^\C:\,|\al|=1\big\}>0.\end{equation}
\end{lemma}
\begin{proof}[Proof of Lemma \ref{le:U}] It follows from hypothesis (H) that there exists $\de_0>0$ such that $\mu^{-1}(\bar B_{\de_0})\sub M^*$. We define
\begin{eqnarray}\nn&C:=\sup\big\{|[\xi,\eta]|\,\big|\,\xi,\eta\in\g:\,|\xi|\leq1,\,|\eta|\leq 1\big\},&\\
\nn&c_0:=\inf\big\{|L_x\xi|/|\xi|\,\big|\,x\in \mu^{-1}(\bar B_{\de_0}),\,0\neq\xi\in\g\big\}.
\end{eqnarray}
Since the action of $G$ on $M^*$ is free, it follows that $L_x:\g\to T_xM$ is injective, for $x\in M^*$. Furthermore, by hypothesis (H) the set $\mu^{-1}(\bar B_{\de_0})$ is compact. It follows that $c_0>0$. We choose a positive number $\de<\min\{\de_0,c_0/C,c_0^3/C\}$, and we define $U:=\mu^{-1}(B_\de)$. 
\begin{claim}\label{claim:c inf} Inequality (\ref{eq:c inf}) holds.
\end{claim}
\begin{pf}[Proof of Claim \ref{claim:c inf}] Let $x\in U$ and $\al=\xi+i\eta\in\g^\C$. Then 
\begin{equation}\label{eq:d mu x L x}d\mu(x)L_x^\C\al=[\mu(x),\xi]+L_x^*L_x\eta.\end{equation} 
Using the last assertion in (\ref{eq:L x * L x c}), we have 
\begin{equation}\label{eq:Pr x L x}{\Pr}_xL_x^\C\al=L_x\xi-L_x(L_x^*L_x)^{-1}[\mu(x),\eta].\end{equation}
By the first assertion in (\ref{eq:L x * L x c}), we have $|L_x^*L_x\eta|\geq c_0^2|\eta|$. Combining this with (\ref{eq:d mu x L x},\ref{eq:Pr x L x}) and the second assertion in (\ref{eq:L x * L x c}), we obtain
\[\big|d\mu(x)L_x^\C\al\big|+|\Pr L_x^\C\al|\geq -C\de|\xi|+c_0^2|\eta|+c_0|\xi|-c_0^{-1}C\de|\eta|.\]
Inequality (\ref{eq:c inf}) follows now from our choice of $\de$. This proves Claim \ref{claim:c inf} and completes the proof of Lemma \ref{le:U}.
\end{pf}
\end{proof}
The following lemma was mentioned in Section \ref{sec:main}. Recall the definition (\ref{eq:BB p lam}) of $\BB^{p,\lam}$, and that $\G^{2,p}_\loc(P)$ denotes the group of gauge transformations on $P$ of class $W^{2,p}_\loc$. 
\begin{lemma}\label{le:free} For $p>2$ and $\lam>1-2/p$ $\G^{2,p}_\loc(P)$ acts freely on $\BB^{p,\lam}$.
\end{lemma}
\begin{proof}[Proof of Lemma \ref{le:free}]\setcounter{claim}{0} Let $w:=(u,A)\in\BB^{p,\lam}$ and $g\in\G^{2,p}_\loc(P)$ be such that $g_*w=w$. It follows from hypothesis (H) that there exists $\de>0$ such that $\mu^{-1}(\bar B_\de)\sub M^*$ (defined as in (\ref{eq:M *})). Furthermore, Lemma \ref{le:si} below implies that there exists $R>0$ such that $|\mu\circ u(p)|<\de$, for $p\in \pi^{-1}(\R^2\wo B_R)$. We choose $p_0\in \pi^{-1}(R)\sub P$. Since $g(p_0)u(p_0)=u(p_0)$ and $u(p_0)\in M^*$, it follows that $g(p_0)=\one$. Let $p_1\in P$. We choose a smooth path $p:[0,1]\to P$ such that $p(i)=p_i$, for $i=0,1$. Then the map $g_p:=g\circ p:[0,1]\to G$ solves the ordinary differential equation $\dot g_p=g_pA\dot p-(A\dot p)g_p$, $g_p(0)=\one$. It follows that $g_p\const \one$, and hence $g(p_1)=\one$. This proves Lemma \ref{le:free}.
\end{proof}
We now prove Proposition \ref{prop:P}. Let $M,\om,G,,\g,\lan\cdot,\cdot\ran_\g,\mu$ and $J$ be as in Section \ref{sec:main}, $\Si:=\R^2,\om_\Si:=\om_0,j:=i,P\to\R^2$ a principal $G$-bundle, $p>2,\lam>1-2/p$ and $w\in\BB^p_\lam$. Assume that hypothesis (H) holds.
\begin{lemma}\label{le:si} There exists a smooth section $\si$ of the restriction of the bundle $P$ to $B_1^C$ and a point $x_\infty\in\mu^{-1}(0)$, such that $u\circ \si(re^{i\phi})$ converges to $x_\infty$, uniformly in $\phi\in\R$, as $r\to\infty$, and $\si^*A\in L^p_\lam(B_1^C)$. 
\end{lemma}
\begin{proof}[Proof of Lemma \ref{le:si}]\label{proof:le:si}\setcounter{claim}{0} 
\begin{claim}\label{claim:R u} The expression $|\mu\circ u|(re^{i\phi})$ converges to 0, uniformly in $\phi\in\R$, as $r\to\infty$. 
\end{claim}
\begin{proof}[Proof of Claim \ref{claim:R u}] We define the function $f:=|\mu\circ u|^2:M\to\R$. It follows from the ad-invariance of $\lan\cdot,\cdot\ran_\g$ that 
\begin{equation}\label{eq:d mu u 2}df=2\big\lan d_A(\mu\circ u),\mu\circ u\big\ran_\g=2\big\lan d\mu(u)d_Au,\mu\circ u\big\ran_\g.\end{equation}
Since $\BAR{u(P)}\sub M$ is compact, we have $\sup_{\R^2}|d\mu(u)|<\infty,\sup_{\R^2}|\mu\circ u|<\infty$. Furthermore, $|d_Au|\leq \sqrt{e_w}\in L^p_\lam$. Combining this with (\ref{eq:d mu u 2}), it follows that $df\in L^p_\lam$. Therefore, by Proposition \ref{prop:Hardy} (Hardy-type inequality, applied with $u$ replaced by $f$) the expression $f(re^{i\phi})$ converges to some number $y_\infty\in \R$, as $r\to\infty$, uniformly in $\phi\in\R$. Since $|\mu\circ u|\leq\sqrt{e_w}\in L^p_\lam$, it follows that $y_\infty=0$. 
This proves Claim \ref{claim:R u}.
\end{proof}
It follows from hypothesis (H) that there exists a $\de>0$ such that $\mu^{-1}(\bar B_\de)\sub M^*$ (defined as in (\ref{eq:M *})). We choose $R>0$ so big that $|\mu\circ u|(z)\leq\de$ if $z\in B_{R-1}^C=\R^2\wo B_{R-1}$. Since $G$ is compact, the action of it on $M$ is proper. Hence the local slice theorem implies that $M^*/G$ carries a unique manifold structure such that the canonical projection $\pi^{M^*}:M^*\to M^*/G$ is a submersion. Consider the map $\bar u:B_{R-1}^C\to M^*/G$ defined by $\bar u(z):=G u(p)$, where $p\in \pi^{-1}(z)\sub P$ is arbitrary. 
\begin{claim}\label{claim:bar u} The point $\bar u(re^{i\phi})$ converges to some point $\bar x_\infty\in \mu^{-1}(0)/G$, uniformly in $\phi\in\R$, as $r\to\infty$. 
\end{claim}
\begin{proof}[Proof of Claim \ref{claim:bar u}] We choose $n\in\N$ and an embedding $\iota:M^*/G\to \R^n$. Furthermore, we choose a smooth function $\rho:\R^2\to\R$ that vanishes on $B_{R-1}$ and equals 1 on $B_R^C$. We define $f:\R^2\to \R^n$ to be the map given by $\rho \cdot\iota\circ\bar u$ on $B_{R-1}^C$ and by 0 on $B_{R-1}$. It follows that $\Vert df\Vert_{L^p_\lam(B_R^C)}\leq\big\Vert d\iota(\bar u)d\bar u\big\Vert_{L^p_\lam(B_R^C)}+\Vert (d\rho)\iota\circ\bar u\Vert_{L^p_\lam(B_R\wo B_{R-1})}$. Furthermore, 
\begin{equation}\nn\label{eq:d iota}\big\Vert d\iota(\bar u)d\bar u\big\Vert_{L^p_\lam(B_R^C)}\leq\Vert d\iota(\bar u)\Vert_{L^\infty(B_R^C)}\Vert d_Au\Vert_{L^p_\lam(B_R^C)}.
\end{equation}
Our assumption $w=(u,A)\in\BB^p_\lam$ implies that $\Vert d_Au\Vert_{L^p_\lam(B_R^C)}<\infty$. Furthermore, $\mu$ is proper by the hypothesis (H), hence the set $\mu^{-1}(\bar B_\de)$ is compact. Thus the same holds for the set $\pi^{M^*}(\mu^{-1}(\bar B_\de))$. This set contains the image of $\bar u$. It follows that $\Vert d\iota(\bar u)\Vert_{L^\infty(B_R^C)}<\infty$, and therefore $\Vert df\Vert_{L^p_\lam(\R^2)}\leq\Vert df\Vert_{L^p_\lam(B_R)}+\Vert df\Vert_{L^p_\lam(B_R^C)}<\infty$. Hence the hypotheses of Proposition \ref{prop:Hardy} are satisfied. It follows that the point $f(re^{i\phi})$ converges to some point $y_\infty\in\R^n$, uniformly in $\phi\in\R$, as $r\to\infty$. Claim \ref{claim:bar u} follows.
\end{proof}
Let $\bar x_\infty$ be as in Claim \ref{claim:bar u}. We choose a local slice around $\bar x_\infty$, \ie a pair $(\bar U,\wt\si)$, where $\bar U\sub M^*/G$ is an open neighborhood of $\bar x_\infty$, and $\wt\si:\bar U\to M^*$ is a smooth map satisfying $\pi^{M^*}\circ \wt\si=\id_{\bar U}$. Then there exists a unique section $\si'$ of $P|_{B_R^C}$ such that $\wt\si\circ\bar u=u\circ\si'$. By the homotopy lifting property of $P$ we may extend this to a continuous section $\si''$ of $P|_{B_1^C}$. Smoothing out $\si''$ on $B_{R+1}\wo B_1$, we obtain a smooth section $\si$ of $P|_{B_1^C}$. We define $x_\infty:=\wt\si(\bar x_\infty)$. It follows from Claim \ref{claim:bar u} that $u\circ\si(re^{i\phi})$ converges to $x_\infty$, uniformly in $\phi$, for $r\to\infty$. Furthermore, 
\[\Vert L_{u\circ \si}\si^*A\Vert_{L^p_\lam(B_{R+1}^C)}\leq \Vert du\,d\si'\Vert_{L^p_\lam(B_{R+1}^C)}+\Vert d_{A\,d\si'}u\Vert_{L^p_\lam(B_{R+1}^C)},\]
\[du\,d\si'=d(u\circ\si')=d\wt\si\,d\bar u,\quad |d\bar u|\leq |d_Au|.\]
Since $\inf\big\{|L_u(p)\xi|\,\big|\,p\in P|_{B_{R+1}^C},\,\xi\in\g:\,|\xi|=1\big\}>0$ and $\Vert d_Au\Vert_{p,\lam}<\infty$, it follows that $\si^*A\in L^p_\lam(B_1^C)$. This proves Lemma \ref{le:si}.
\end{proof}
\begin{proof}[Proof of Proposition \ref{prop:P}]\setcounter{claim}{0}\label{proof:prop:P} We choose $\si$ and $x_\infty$ as in Lemma \ref{le:si}, and we define $\wt P$ to be the quotient of $P\disj (S^2\wo\{0\})\x G$ under the equivalence relation generated by $p\sim (\pi(p),g)$, where $g\in G$ is determined by $\si(z)g=p$, for $p\in P$. We define $\wt u:\wt P\to M$ by $\wt u([p]):=u(p)$, for $p\in P$, and $\wt u([\infty,g]):=g^{-1}x_\infty$, for $g\in G$. If follows from the statement of Lemma \ref{le:si} that this map is continuous. Let now $\wt P_1$ and $\wt P_2$ be two extensions of $P$, for which the map $u$ extends to continuous maps $\wt u_i:\wt P_i\to M$. We define $\Psi:\wt P_1\to\wt P_2$ to be the identity on $P\sub\wt P_1$, and for $p$ in the fiber of $\wt P_1$ over $\infty$ we define $\Psi(p)$ to be the unique point $p'\in \wt P_2$ such that $\wt u_1(p)=\wt u_2(p')$. (It follows from Lemma \ref{le:si} that this point is unique.) This map has the required properties. This proves Proposition \ref{prop:P}.
\end{proof}

\end{document}